\documentclass[12pt]{article}

\usepackage{amsmath, amsthm, amssymb, amsfonts}
\usepackage{mathrsfs}    
\usepackage{bm}          
\usepackage{graphicx}     
\usepackage{hyperref}     
\usepackage{enumitem}     
\usepackage{tikz-cd}      
\usepackage{geometry}     
\usepackage{tikz}
\usepackage[all]{xy}    
\newtheorem{theorem}{Theorem}[section]
\newtheorem{lemma}[theorem]{Lemma}
\newtheorem{corollary}[theorem]{Corollary}
\newtheorem{proposition}[theorem]{Proposition}
\theoremstyle{definition}

\theoremstyle{remark}
\newtheorem{remark}[theorem]{Remark}
\newcommand{\subjclass}[2][]{\medskip\noindent\textbf{Mathematics Subject Classification (2020):} #2.}
\newcommand{\keywords}[1]{\medskip\noindent\textbf{Keywords:} #1.}
\title{Koszul Duality for Quadratic Monomial Algebras}
\author{A. M. Bouhada\\
\small \texttt{}}
\date{}

\begin{document}

\maketitle

\begin{abstract} Let \(\Lambda\) be a quadratic monomial algebra and let \(\Lambda^{!}\) denote its Koszul dual. We study finitely presented and finitely copresented graded modules over \(\Lambda\) and \(\Lambda^{!}\), and use this structure to obtain explicit forms of derived and singular Koszul duality. We prove that quadratic monomial algebras are both graded coherent and graded cocoherent. Applying this to the Koszul dual \(\Lambda^{!}\), which is again quadratic monomial, we show that finitely presented and finitely copresented graded \(\Lambda^{!}\)-modules coincide with perfect and coperfect modules, respectively. It follows that the corresponding tails and cotails categories are hereditary abelian categories, and their objects are described in terms of linear and colinear modules. We also prove that quadratic monomial algebras are absolutely Koszul and have global linearity defect at most two. Consequently, finite-dimensional graded \(\Lambda\)-modules have rational Poincar\'e series, while finitely presented and finitely copresented graded \(\Lambda^{!}\)-modules have rational Hilbert series. As applications, we refine graded and ungraded derived Koszul dualities, singular Koszul dualities, and Bernstein--Gelfand--Gelfand correspondences for quadratic monomial algebras. We further describe the induced nonstandard \(t\)-structures on bounded derived and singularity categories and derive bounds for the finitistic dimension of \(\Lambda\) in terms of finite paths in the Koszul dual algebra. \end{abstract} \subjclass[2020]{16E30, 18G35, 18E30, 16G20, 16E45} \keywords{Koszul duality, quadratic monomial algebras, coherent algebras, cocoherent algebras, derived categories, singularity categories, BGG correspondence, \(t\)-structures, finitistic dimension}
\begin{center}
\tableofcontents
\end{center}

\section{Introduction}

Quadratic monomial algebras form a distinguished class of graded algebras which arise naturally in several areas of mathematics, including the representation theory of finite-dimensional algebras, the theory of semisimple Lie algebras, and homological mirror symmetry. In the representation theory of finite-dimensional algebras, they provide a rich source of examples and counterexamples for testing homological conjectures and for studying subtle homological phenomena. For instance, for monomial algebras the big finitistic dimension need not coincide with the small finitistic dimension, whereas the latter is always finite; see~\cite{13,33}. 

\medskip 

Quadratic monomial algebras already occur in classical representation-theoretic settings. For example, the principal block \(\mathcal{O}_{0}\) of the Bernstein--Gel'fand--Gel'fand category \(\mathcal{O}\) associated with \(\mathfrak{sl}_{2}(\mathbb{C})\) is equivalent to the category of finite-dimensional modules over the algebra given by the quiver \[ \begin{tikzpicture}[->,>=stealth,auto,scale=0.7,baseline=-0.5ex] \node (1) at (0,0) {$1$}; \node (2) at (1.5,0) {$2$}; \draw[bend left=40] (1) to node[above] {$\alpha$} (2); \draw[bend left=40] (2) to node[below] {$\beta$} (1); \end{tikzpicture} \qquad \text{with relation } \alpha\beta=0; \] see~\cite[Theorem~5.3.1]{25}. They also appear naturally in homological mirror symmetry through their connections with Fukaya categories; see \cite{4,15,22}. 

\medskip 

Both in the commutative setting~\cite{12} and in the noncommutative setting~\cite{14}, quadratic monomial algebras are Koszul; see also Theorem~2.3.  Moreover, the Koszul dual of a quadratic monomial algebra is again quadratic monomial. This stability under Koszul duality makes quadratic monomial algebras particularly well suited for the study of derived and singular Koszul dualities. 

\medskip 

Despite their importance, bounded derived categories associated with quadratic monomial algebras are not well understood in full generality. Apart from the general framework developed in~\cite{5}, explicit descriptions remain largely unavailable. Much of the existing literature therefore focuses on special subclasses, often motivated by their connections with Fukaya categories and homological mirror symmetry; see, for example,~\cite{15,22}. 

\medskip 

The purpose of this paper is to give a systematic study of quadratic monomial algebras from the perspective of Koszul duality. More precisely, this work continues the program initiated in~\cite{5}, where several classes of algebras were shown to admit refined forms of derived and singular Koszul duality. We show that quadratic monomial algebras exhibit particularly favorable duality properties within this framework. 

\medskip 

Our approach is based on a detailed analysis of finitely presented and finitely copresented graded modules over the Koszul dual algebra \(\Lambda^{!}\). We prove that, for every finite-dimensional quadratic monomial algebra \(\Lambda\), the algebra \(\Lambda^{!}\) is both graded coherent and graded co-coherent. Moreover, finitely presented graded modules coincide with perfect modules, and finitely copresented graded modules coincide with coperfect modules. As a consequence, the associated tails and cotails categories are hereditary abelian categories and admit explicit structural descriptions. 

\medskip 

We also prove that finite-dimensional quadratic monomial algebras are absolutely Koszul and satisfy \(\operatorname{gl.ld}\Lambda \leq 2.\) In particular, every finitely presented graded module has rational Poincar\'e series and rational Hilbert series. We further show that finitely presented and finitely copresented graded modules are precisely those admitting linear and colinear truncations, respectively. \medskip Using these structural results, we refine the graded derived and singular Koszul dualities for quadratic monomial algebras. More precisely, we obtain explicit triangulated equivalences relating bounded derived and singular categories of graded modules over \(\Lambda\) and over its Koszul dual \(\Lambda^{!}\). We further show that these dualities induce nonstandard \(t\)-structures on the bounded derived and singularity categories. The corresponding hearts are described explicitly in terms of linear complexes and cohomological shifts of linear modules. 

\medskip 
As a final application, we study the behavior of the finitistic dimension under Koszul duality. We show how finiteness properties of the Koszul dual algebra control the finitistic dimension of finite-dimensional quadratic monomial algebras. This suggests a possible route toward applications to the finitistic dimension conjecture for finite-dimensional Koszul algebras. 

\medskip 

\noindent 

\textbf{Organization.} Section~2 recalls the necessary background on quadratic monomial algebras and their homological properties. In Section~3, we establish the main structural results on finitely presented and finitely copresented graded modules. We then derive the refined Koszul dualities and apply them to bounded derived categories, singularity categories, and finitistic dimensions.
\medskip

\section{Preliminaries}
The aim of this section is to recall the basic definitions and properties used throughout the paper. 
We begin by recalling that a \emph{\(k\)-linear category} is a category whose Hom-sets are \(k\)-vector spaces and whose composition maps are bilinear over \(k\). 
Throughout, \(\Lambda\) denotes a quadratic monomial algebra, which may be either locally finite-dimensional or finite-dimensional, and is regarded as a \(k\)-linear category. 
Unless otherwise specified, all \(\Lambda\)-modules are assumed to be left modules, viewed as covariant \(k\)-linear functors.

We first recall the definition of a quadratic monomial algebra. In this paper, a \emph{quadratic monomial algebra} is an algebra of the form \(\Lambda = kQ/I,\) where \(Q\) is a finite quiver and \(I\subseteq kQ\) is an ideal generated by quadratic monomials, that is, by paths of length two. We say that \(\Lambda\) is \emph{locally finite-dimensional} if each graded component \(\Lambda_i\) is finite-dimensional for all \(i\geq 0\). We say that \(\Lambda\) is \emph{finite-dimensional} if, in addition, the ideal \(I\) is admissible.

A \emph{left \(\Lambda\)-module} is a covariant \(k\)-linear functor
\[
M \colon \Lambda \longrightarrow k\textup{-Mod},
\]
where \(k\textup{-Mod}\) denotes the category of \(k\)-vector spaces and \(k\)-linear maps. 
Explicitly, such a functor assigns to each vertex \(x \in Q_0\) a vector space \(M(x)\), and to each arrow \(\alpha : x \to y\) a linear map 
\[
M(\alpha) \colon M(x) \longrightarrow M(y)
\]
satisfying the usual compatibility with composition. 
Morphisms of \(\Lambda\)-modules are natural transformations. When \(\Lambda\) is finite-dimensional, we denote by
\( 
\Lambda\textup{-mod}
\)
the category of finite-dimensional left \(\Lambda\)-modules. We denote by
\( 
\Lambda\textup{-proj} \)
and
\( \Lambda\textup{-inj}
\)
the full subcategories of finite-dimensional projective and injective
\(\Lambda\)-modules, respectively.

\medskip

We now regard \(\Lambda\) as a \emph{graded \(k\)-linear category}, that is, a category whose Hom-spaces are \(\mathbb{Z}\)-graded \(k\)-vector spaces and whose composition maps are bilinear and compatible with the grading. 
A \emph{graded left \(\Lambda\)-module} is a covariant \(k\)-linear functor
\[
N \colon \Lambda \longrightarrow k\textup{-GMOD},
\]
where \(k\textup{-GMOD}\) denotes the category of  \(\mathbb{Z}\)-graded \(k\)-vector spaces and degree-preserving \(k\)-linear maps. 
Explicitly, such a functor assigns to each vertex \(x \in Q_0\) a graded vector space
\[
N(x) = \bigoplus_{i \in \mathbb{Z}} N_i(x),
\]
and to each arrow \(\alpha : x \to y\) a homogeneous \(k\)-linear map
\[
N(\alpha) \colon N(x) \longrightarrow N(y)
\]
of degree \(0\), compatible with composition in \(\Lambda\). 
Morphisms of graded \(\Lambda\)-modules are natural transformations preserving the grading.

\medskip

A graded \(\Lambda\)-module \(N\) is said to be:
\begin{itemize}
    \item \emph{left bounded} if \(N_i(x) = 0\) for all \(x \in Q_0\) and all sufficiently small \(i \in \mathbb{Z}\);
    
    \item \emph{right bounded} if \(N_i(x) = 0\) for all \(x \in Q_0\) and all sufficiently large \(i \in \mathbb{Z}\);
    
    \item \emph{locally finite-dimensional} if \(N_i(x)\) is finite-dimensional for all \(x \in Q_0\) and all \(i \in \mathbb{Z}\);
    
    \item \emph{bounded} if it is both left and right bounded. In this case, \(N\) is concentrated in finitely many degrees and is therefore finite-dimensional.
\end{itemize}

\medskip 

We denote by \(\Lambda\textup{-GMod}\) the category of locally finite-dimensional graded \(\Lambda\)-modules, and by
\(\Lambda\textup{-gmod}\) its full subcategory of finite-dimensional graded modules. Both categories \(\Lambda\textup{-GMod}\) and \(\Lambda\textup{-gmod}\) are equipped with the grading shift functor \(\langle i\rangle\), \(i \in \mathbb{Z}\), defined by
\(M\langle i\rangle_{n} = M_{n-i},\)
for any graded module \(M\). We write
\(\Lambda\textup{-GMod}^{+}\) (resp.\ \(\Lambda\textup{-GMod}^{-}\))
for the full subcategory of \(\Lambda\textup{-GMod}\) consisting of left bounded
(resp.\ right bounded) modules.

We denote by \(\Lambda\textup{-Proj}^{\mathbb{Z}}\) the full subcategory of
\(\Lambda\textup{-GMod}\) consisting of finite direct sums of grading shifts
\(P_x\langle i\rangle\), where \(x\in Q_0\) and \(i\in\mathbb{Z}\).
For each vertex \(x\in Q_0\), the indecomposable graded projective
\(\Lambda\)-module \(P_x\) is the representable functor
\[
P_x \colon \Lambda \longrightarrow k\textup{-GMod}, 
\qquad 
P_x(y)=\Lambda(x,y),
\]
endowed with the grading induced by path length. Thus, for each
\(n\ge 0\), the homogeneous component \((P_x(y))_n\) is spanned by all
paths of length \(n\) from \(x\) to \(y\).
For an arrow \(\alpha\colon a\to b\), the induced morphism
\[
P_x(\alpha)\colon P_x(a)\longrightarrow P_x(b)
\]
is given by right multiplication \(p \mapsto p\alpha\).

\medskip

If \(\Lambda\) is finite-dimensional, we write
\(\Lambda\textup{-proj}^{\mathbb{Z}}\) for the full subcategory of
finite-dimensional graded projective \(\Lambda\)-modules.

\medskip

Dually, we denote by \(\Lambda\textup{-Inj}^{\mathbb{Z}}\) the full subcategory
of \(\Lambda\textup{-GMod}\) consisting of finite direct sums of grading shifts
\(I_x\langle i\rangle\), where \(x\in Q_0\) and \(i\in\mathbb{Z}\).
For each vertex \(x\in Q_0\), the indecomposable graded injective
\(\Lambda\)-module \(I_x\) is given by the functor
\[
I_x \colon \Lambda \longrightarrow k\textup{-GMod}, 
\qquad 
I_x(y)=D\Lambda(y,x),
\]
where \(D=\operatorname{Hom}_k(-,k)\).
Equivalently, \(I_x \cong D\!\bigl(\Lambda^{\operatorname{op}}(x,-)\bigr)\).

The grading on \(I_x(y)\) is induced by path length, so that
\((I_x(y))_n\) is spanned by the duals of paths of length \(n\)
from \(y\) to \(x\).
For an arrow \(\alpha\colon a\to b\), the induced morphism
\[
I_x(\alpha)\colon I_x(a)\longrightarrow I_x(b)
\]
is given by precomposition with \(\alpha\); explicitly, for
\(f\in D\Lambda(a,x)\),
\[
I_x(\alpha)(f)(p)=f(p\alpha)
\qquad \text{for all } p\in \Lambda(b,x).
\]

Equivalently, under the identification of \(I_x(y)\) with the space spanned
by paths from \(y\) to \(x\), the map \(I_x(\alpha)\) is given by
\[
q \longmapsto
\begin{cases}
q' & \text{if } q = q'\alpha,\\
0 & \text{otherwise}.
\end{cases}
\]

\medskip

If \(\Lambda\) is finite-dimensional, we write
\(\Lambda\textup{-inj}^{\mathbb{Z}}\) for the full subcategory of
finite-dimensional graded injective \(\Lambda\)-modules.

\medskip 
We write \( \Lambda\textup{-Fp}^{\mathbb{Z}} \) for the category of finitely presented \emph{graded} \( \Lambda \)-modules, that is, modules \( M \) admitting a projective presentation
\[
   P^{1} \longrightarrow P^{0} \longrightarrow M \longrightarrow 0,
\]
with \( P^{1}, P^{0} \) finite direct sums of objects of \( \Lambda\textup{-Proj}^{\mathbb{Z}} \).

Dually, we denote by \( \Lambda\textup{-Fcp}^{\mathbb{Z}} \) the category of finitely copresented \emph{graded} \( \Lambda \)-modules, that is, modules \( N \) admitting an injective copresentation
\[
   0 \longrightarrow N \longrightarrow I^{0} \longrightarrow I^{1},
\]
with \( I^{0}, I^{1} \) finite direct sums of objects of \( \Lambda\textup{-Inj}^{\mathbb{Z}} \).

\medskip

The algebra \( \Lambda \) is said to be \emph{left coherent} (respectively, \emph{left co-coherent}) if the category 
\( \Lambda\textup{-Fp}^{\mathbb{Z}} \) (respectively, \( \Lambda\textup{-Fcp}^{\mathbb{Z}} \)) is abelian.

When \( \Lambda \) is left coherent, one defines the \emph{tails category}
\(\Lambda\textup{-Fp}^{\mathbb{Z}} \big/ \Lambda\textup{-gmod},\)
which plays a central role in noncommutative projective geometry as a noncommutative analogue of the category \( \operatorname{Coh}(X) \) of coherent sheaves on a projective variety \( X \); see~\cite{1}.

Dually, when \( \Lambda \) is left co-coherent, one defines the \emph{cotails category}
\(\Lambda\textup{-Fcp}^{\mathbb{Z}} \big/ \Lambda\textup{-gmod},\)
which may be viewed as the dual counterpart of the tails category.
\medskip

Even when \( \Lambda \) is not coherent (respectively, not co-coherent), one can still consider the exact categories
\[
   \Lambda\textup{-Cop}^{\mathbb{Z}} \big/ \Lambda\textup{-gmod} 
   \quad \text{and} \quad
   \Lambda\textup{-Pe}^{\mathbb{Z}} \big/ \Lambda\textup{-gmod},
\]
which remain fundamental in the formulation of graded singular Koszul dualities and the Bernstein--Gelfand--Gelfand correspondence; see~\cite{5}.

\medskip 

We denote by \( \Lambda\textup{-Pe}^{\mathbb{Z}} \) the category of \emph{perfect graded modules}, that is, graded \( \Lambda \)-modules \( M \) admitting a finite projective resolution
\[
   0 \longrightarrow P^{n} \longrightarrow P^{n+1} \longrightarrow \cdots \longrightarrow P^{1} \longrightarrow P^{0} \longrightarrow M \longrightarrow 0,
\]
where each \( P^{i} \) is a finite direct sum of objects from \( \Lambda\textup{-Proj}^{\mathbb{Z}} \).

Dually, we denote by \( \Lambda\textup{-Cop}^{\mathbb{Z}} \) the category of \emph{coperfect graded modules}, that is, graded \( \Lambda \)-modules \( N \) admitting a finite injective coresolution
\[
   0 \longrightarrow N \longrightarrow I^{0} \longrightarrow I^{1} \longrightarrow \cdots \longrightarrow I^{n}
   \longrightarrow 0,
\]
where each \( I^{i} \) is a finite direct sum of objects from \( \Lambda\textup{-Inj}^{\mathbb{Z}} \).

\medskip 

Given a quadratic monomial algebra \(\Lambda\), we denote by \(\Lambda^!\)
its quadratic dual. Since \(\Lambda = kQ/I\) is quadratic monomial, its
quadratic dual admits the description
\(\Lambda^! = k Q^{\mathrm{op}} / I^!,\)
where \(I^!\) is the ideal generated by quadratic monomials
\(\alpha^{\mathrm{op}} \beta^{\mathrm{op}}\) such that \(\beta \alpha \notin I\).
In particular, \(\Lambda^!\) is again a quadratic monomial algebra.

\medskip

We say that \(\Lambda\) is \emph{Koszul} if, for every \(x \in Q_0\), the
minimal graded projective resolution of the simple module \(S_x\)
\[
\cdots \longrightarrow P^{2} \longrightarrow P^{1}
\longrightarrow P^0 \longrightarrow S_x \longrightarrow 0
\]
is \emph{linear}, that is, each term \(P^{i}\) is generated in internal
degree \(i\). Equivalently, for each \(i \ge 0\),
\[
P^{i} \;\cong\; \bigoplus_{y \in Q_0} P_y\langle i\rangle,
\]
where the direct sum is finite. Such a resolution is called a
\emph{linear resolution}. Although it is classical that quadratic monomial algebras are Koszul, we provide a new proof tailored to our framework; see Theorem~2.3.

\medskip

Similarly, one defines the notions of Koszul and coKoszul modules.
A graded \(\Lambda\)-module (resp.\ \(\Lambda^!\)-module) \(M\) is called
\emph{Koszul} (resp.\ \emph{coKoszul}) if it admits a finite linear
projective resolution (resp.\ a finite colinear injective coresolution).

There is a standard homological characterization of Koszul and coKoszul
modules. A graded \(\Lambda\)-module \(M\) is Koszul if
\[
\operatorname{Ext}^n_{\Lambda\textup{-GMod}}\bigl(M, S_x\langle i\rangle\bigr)=0
\quad \text{for all } i \neq n \text{ and all } x \in Q_0,
\]
and a graded \(\Lambda^!\)-module \(M\) is coKoszul if
\[
\operatorname{Ext}^n_{\Lambda^!\textup{-GMod}}\bigl(S^{!}_x\langle -i\rangle, M\bigr)=0
\quad \text{for all } i \neq n \text{ and all } x \in Q_0.
\]
Here \(S_x^{!}\) denotes the simple graded \(\Lambda^{!}\)-module
corresponding to the vertex \(x\).

\medskip 

We say that a graded module \(M\) is \emph{linear} (resp.\ \emph{colinear})
if there exists an integer \(i\) such that \(M\langle -i\rangle\) is Koszul
(resp.\ coKoszul).

\medskip

Recall that the algebra \(\Lambda\) is called \emph{Iwanaga--Gorenstein}
if it has finite injective dimension on both sides, that is,
\[
\operatorname{id}_{\Lambda}\Lambda < \infty
\qquad\text{and}\qquad
\operatorname{id}_{\Lambda^{\mathrm{op}}}\Lambda < \infty .
\]

A finite-dimensional \(\Lambda\)-module \(M\) is called
\emph{Gorenstein-projective} if there exists a totally acyclic complex
\(P^\bullet\) of finitely generated projective \(\Lambda\)-modules such that
\[
M \cong \ker(P^0 \longrightarrow P^1),
\]
where \(P^\bullet\) is acyclic and
\(\operatorname{Hom}_{\Lambda}(P^\bullet,\Lambda)\) is also acyclic.

We denote by
\(\Lambda\textup{-Gproj}\)
the category of finite-dimensional Gorenstein-projective \(\Lambda\)-modules.
Its stable category is denoted by
\(\Lambda\textup{-}\underline{\mathrm{Gproj}}\)
and is obtained by factoring out morphisms that factor through projective modules.

Similarly, a finite-dimensional graded \(\Lambda\)-module \(M\) is called
\emph{graded Gorenstein-projective} if it admits a totally acyclic complex
\(P^\bullet\) of finitely generated graded projective \(\Lambda\)-modules such that
\[
M \cong \ker(P^0 \longrightarrow P^1).
\]
We denote the corresponding category by
\(\Lambda\textup{-Gproj}^{\mathbb Z},\)
and its stable category by
\(\Lambda\textup{-}\underline{\mathrm{Gproj}}^{\mathbb Z}.\)

\medskip 

Given an exact category \(\mathcal{A}\) in the sense of Quillen with enough projective objects, we denote by \(\mathsf{D}^{b}(\mathcal{A})\) its bounded derived category and by
\[
\mathsf{D}_{\mathrm{sg}}(\mathcal{A})
:=
\mathsf{D}^{b}(\mathcal{A})/
\mathsf{K}^{b}(\operatorname{Proj}\mathcal{A})
\]
its singularity category.

\medskip 

We now recall the Koszul duality theorems established in~\cite{5}, which provide a unified description of the bounded derived and singularity categories in both the graded and ungraded settings.

\begin{theorem}[\cite{5}]
Let \(\Lambda\) be a finite-dimensional Koszul algebra. Then the following equivalences of triangulated categories hold.
\begin{enumerate}\itemsep=0.4em

\item[\textup{(i)}] \textbf{Graded derived Koszul duality:}
\[
\mathfrak{F} \colon
\mathsf{D}^{b}\bigl(\Lambda^{!}\textup{-Cop}^{\mathbb{Z}}\bigr)
\;\xrightarrow{\ \sim\ }\;
\mathsf{D}^{b}\bigl(\Lambda\textup{-gmod}\bigr),
\qquad
\mathfrak{F} \colon
\mathsf{D}^{b}\bigl(\Lambda\textup{-gmod}\bigr)
\;\xrightarrow{\ \sim\ }\;
\mathsf{D}^{b}\bigl(\Lambda^{!}\textup{-Pe}^{\mathbb{Z}}\bigr).
\]

In particular, if \(\Lambda\) has finite global dimension, then the graded derived Koszul duality induces a triangulated equivalence
\[
\mathsf{D}^{b}\bigl(\Lambda^{!}\textup{-gmod}\bigr)
\;\xrightarrow{\ \sim\ }\;
\mathsf{D}^{b}\bigl(\Lambda\textup{-gmod}\bigr).
\]

\item[\textup{(ii)}] \textbf{Graded singular Koszul duality:}
\[
\mathsf{D}^{b}\!\Bigl(\Lambda^{!}\textup{-Cop}^{\mathbb{Z}}\big/\Lambda^{!}\textup{-gmod}\Bigr)
\;\xrightarrow{\ \sim\ }\;
\mathsf{D}_{\mathrm{sg}}\bigl(\Lambda\textup{-gmod}\bigr).
\]

If, in addition, \(\Lambda\) is Iwanaga--Gorenstein, then there is an equivalence
\[
\mathsf{D}^{b}\!\Bigl(\Lambda^{!}\textup{-Pe}^{\mathbb{Z}}\big/\Lambda^{!}\textup{-gmod}\Bigr)
\;\xrightarrow{\ \sim\ }\;
\mathsf{D}^{b}\!\Bigl(\Lambda^{!}\textup{-Cop}^{\mathbb{Z}}\big/\Lambda^{!}\textup{-gmod}\Bigr)
\;\xrightarrow{\ \sim\ }\;
\mathsf{D}_{\mathrm{sg}}\bigl(\Lambda\textup{-gmod}\bigr).
\]

\item[\textup{(iii)}] \textbf{Ungraded derived and singular Koszul dualities:}
\[
\mathrm{H}^{0}\!\Bigl(
\operatorname{pretr}\bigl(
\mathsf{D}_{dg}^{b}(\Lambda^{!}\textup{-Cop}^{\mathbb{Z}})
\big/\langle 1 \rangle[1]
\bigr)
\Bigr)
\;\xrightarrow{\ \sim\ }\;
\mathsf{D}^{b}\bigl(\Lambda\textup{-mod}\bigr),
\]
\[
\mathrm{H}^{0}\!\Bigl(
\operatorname{pretr}\bigl(
\mathsf{D}_{dg}^{b}(\Lambda^{!}\textup{-Cop}^{\mathbb{Z}}/\Lambda^{!}\textup{-gmod})
\big/\langle 1 \rangle[1]
\bigr)
\Bigr)
\;\xrightarrow{\ \sim\ }\;
\mathsf{D}_{\mathrm{sg}}\bigl(\Lambda\textup{-mod}\bigr).
\]

\item[\textup{(iv)}] \textbf{(Graded and ungraded) BGG correspondences:}

If \(\Lambda\) is Iwanaga--Gorenstein, then there are triangulated equivalences
\[
\mathsf{D}^{b}\!\Bigl(\Lambda^{!}\textup{-Pe}^{\mathbb{Z}}\big/\Lambda^{!}\textup{-gmod}\Bigr)
\;\xrightarrow{\ \sim\ }\;
\mathsf{D}^{b}\!\Bigl(\Lambda^{!}\textup{-Cop}^{\mathbb{Z}}\big/\Lambda^{!}\textup{-gmod}\Bigr)
\;\xrightarrow{\ \sim\ }\;
\Lambda\textup{-}\underline{\mathrm{Gproj}}^{\mathbb{Z}}
\]
and, in the ungraded setting,
\[
\mathrm{H}^{0}\!\Bigl(
\operatorname{pretr}\bigl(
\mathsf{D}_{dg}^{b}(\Lambda^{!}\textup{-Cop}^{\mathbb{Z}}/\Lambda^{!}\textup{-gmod})
\big/\langle 1 \rangle[1]
\bigr)
\Bigr)
\;\xrightarrow{\ \sim\ }\;
\Lambda\textup{-}\underline{\mathrm{Gproj}}.
\]

\end{enumerate}
\end{theorem}

Here $\mathrm{H}^{0}\!\bigl(\operatorname{pretr}(-)\bigr)$ denotes the homotopy category of the pretriangulated hull of the corresponding dg orbit category; see~\cite[Subsection~2.4]{5} for further details.

\medskip

One of the main objectives of this paper is to study, in the specific context of quadratic monomial algebras, the Koszul duality equivalences stated above.
\medskip

More precisely, we investigate the categories
\[
\Lambda^{!}\textup{-Cop}^{\mathbb{Z}} 
\quad \text{and} \quad 
\Lambda^{!}\textup{-Cop}^{\mathbb{Z}} \big/ \Lambda^{!}\textup{-gmod}
\quad
(\text{resp.\ } 
\Lambda^{!}\textup{-Pe}^{\mathbb{Z}} 
\quad \text{and} \quad 
\Lambda^{!}\textup{-Pe}^{\mathbb{Z}} \big/ \Lambda^{!}\textup{-gmod}),
\]
with the aim of understanding their structure and their role in the formulation of graded derived and singular Koszul dualities for quadratic monomial algebras.

\medskip

To this end, we recall the construction of the \emph{graded derived Koszul duality}
\[
\mathfrak{F} \colon
\mathsf{D}^{b}\!\bigl(\Lambda^{!}\textup{-Cop}^{\mathbb{Z}}\bigr)
\longrightarrow
\mathsf{D}^{b}\!\bigl(\Lambda\textup{-gmod}\bigr),
\]
which plays a central role in relating the derived categories of a finite-dimensional quadratic monomial algebra and its Koszul dual. 
We briefly review this construction and its main properties, as they will serve as the foundation for the refinements developed in this work.

\medskip

We emphasize that the construction of the functor \(\mathfrak{F}\) extends to the more general setting in which \(\Lambda\) is locally finite-dimensional. However, in order to obtain the graded derived Koszul duality at the level of bounded derived categories, we shall assume throughout that \(\Lambda\) is finite-dimensional.

\medskip

We begin with the following equivalence of abelian categories, which plays a significant role in previous works~\cite{5} (see also~\cite{23,26}) and will be essential in the present context.

\medskip

Recall that a cochain complex of graded projective modules
\[
\cdots \longrightarrow P^{n-1} \longrightarrow P^{n} 
\longrightarrow P^{n+1} \longrightarrow \cdots
\]
is called \emph{linear} if every indecomposable projective direct summand of \( P^{n} \) 
is of the form \( P_{x}\langle n\rangle \), where \( x \in Q_{0} \); 
equivalently, each indecomposable summand is generated in degree \( n \).
We denote by
\(\mathcal{LC}\!\bigl(\Lambda\textup{-proj}^{\mathbb{Z}}\bigr)\)
the category of linear complexes of graded projective \( \Lambda \)-modules.

Martínez-Villa and Saorín~\cite[Theorem~2.4]{23}, and independently Mazorchuk, Ovsienko, and Stroppel~\cite[Theorem~12]{26}, established an equivalence of abelian categories
\[
\Lambda^{!}\textup{-GMod}
\xrightarrow{\;\sim\;}
\mathcal{LC}\!\bigl(\Lambda\textup{-proj}^{\mathbb{Z}}\bigr),
\]
which restricts to an equivalence on finitely generated objects:
\[
\Lambda^{!}\textup{-gmod}
\xrightarrow{\;\sim\;}
\mathcal{LC}^{b}\!\bigl(\Lambda\textup{-proj}^{\mathbb{Z}}\bigr).
\]

These equivalences are realized by a functor
\[
K \colon \Lambda^{!}\textup{-GMod} \longrightarrow
\mathcal{LC}\!\bigl(\Lambda\textup{-proj}^{\mathbb{Z}}\bigr),
\]
which we shall refer to as the \emph{Koszul functor}. It is defined as follows. For \(M \in \Lambda^{!}\textup{-GMod}\) and \(n \in \mathbb{Z}\), set
\[
K(M)^{n} \;:=\; \bigoplus_{x \in Q_{0}} P_{x}\langle n \rangle \otimes_{k} M_{n}(x),
\]
which is a finite direct sum.

The differential
\[
d^{n} \colon K(M)^{n} \longrightarrow K(M)^{n+1}
\]
is given componentwise by
\[
d^{n}_{x,y}
\;=\;
\sum_{\alpha \colon y \to x}
P_{\alpha} \otimes_{k} M(\alpha^{\operatorname{op}}),
\]
where \(P_{\alpha} \colon P_{x} \to P_{y}\) is induced by right multiplication by \(\alpha\), and
\[
M(\alpha^{\operatorname{op}}) \colon M_{n}(x) \longrightarrow M_{n+1}(y)
\]
is the corresponding \(k\)-linear map.

Let \( \Lambda^{!}\textup{-GMod}^{-,b} \) denote the full subcategory of 
\( \Lambda^{!}\textup{-GMod} \) consisting of modules \( M \) such that 
the complex \( K(M) \) is right bounded and has bounded cohomology. 
The above functor then restricts to an equivalence of exact categories
\[
K \colon \Lambda^{!}\textup{-GMod}^{-,b} 
\xrightarrow{\;\sim\;} 
\mathcal{LC}^{-,b}\!\bigl(\Lambda\textup{-proj}^{\mathbb{Z}}\bigr).
\]

Although the category \( \Lambda^{!}\textup{-GMod}^{-,b} \) is not abelian in general, 
it nevertheless admits a natural exact structure that makes it amenable to homological methods. 
In particular, it is closed under extensions and satisfies the two-out-of-three property for short exact sequences. Consequently, it is closed under kernels of admissible epimorphisms and cokernels of admissible monomorphisms.

As shown in~\cite{5}, when \( \Lambda \) is finite-dimensional, the category 
\( \Lambda^{!}\textup{-GMod}^{-,b} \) coincides with \( \Lambda^{!}\textup{-Cop}^{\mathbb{Z}} \), 
the category of coperfect modules.

\medskip 

We shall use the functor \(K\) to give a new proof of the Koszulness of \( \Lambda \) that is fully compatible with our approach to Koszul duality.

\begin{theorem}
The quadratic monomial algebra \( \Lambda \) is Koszul.
\end{theorem}

\begin{proof}
We claim that \(K(I^{!}_{x})\), where \(I^{!}_{x}\) denotes the indecomposable injective
\(\Lambda^{!}\)-module with socle \(S^{!}_{x}\), is a minimal graded projective resolution of the
simple \(\Lambda\)-module \(S_{x}\) generated in degree \(0\).

It suffices to prove exactness in each cohomological degree and to identify \(H^{0}\).
By definition, the degree–\(n\) segment of \(K(I^{!}_{x})\) is
\[
\cdots \longrightarrow
 \bigoplus_{u \in Q_{0}} P_{u}\langle n-1 \rangle \otimes (I^{!}_{x})_{n-1}(u)
 \xrightarrow{d^{\,n+1}}
 \bigoplus_{v \in Q_{0}} P_{v}\langle n \rangle \otimes (I^{!}_{x})_{n}(v)
 \xrightarrow{d^{\,n}}
 \bigoplus_{w \in Q_{0}} P_{w}\langle n+1 \rangle \otimes (I^{!}_{x})_{n+1}(w)
 \longrightarrow \cdots ,
\]
with components
\[
d^{\,n}_{v,w}
\;=\;
\sum_{\alpha \colon w \to v}
\bigl(P_{\alpha} \otimes I^{!}_{x}(\alpha^{\mathrm{op}})\bigr),
\]
where \(P_{\alpha}\colon P_{w}\to P_{v}\) is right multiplication by the arrow \(\alpha\)
and \(I^{!}_{x}(\alpha^{\mathrm{op}})\colon I^{!}_{x}(v)\to I^{!}_{x}(w)\) is the structure map of
\(I^{!}_{x}\) along \(\alpha^{\mathrm{op}}\) in \(\Lambda^{!}\).

Since \(\Lambda\) is quadratic monomial, the nonzero paths form \(k\)-bases of each
\(P_{u}\), and likewise the nonzero \(\Lambda^{!}\)-paths form \(k\)-bases of the
spaces \(I^{!}_{x}(u)\).
Fix vertices \(v,w\) and an elementary tensor
\(p \otimes q^{\mathrm{op}}\in P_{v}\langle n\rangle\otimes I^{!}_{x}(v)\),
where \(p\) is a (nonzero) path in \(\Lambda\) starting at \(v\), and
\(q^{\mathrm{op}}\) is a (nonzero) path in \(\Lambda^{!}\) starting at \(v\) and ending at \(x\).
Then
\[
d^{\,n}_{v,w}(p \otimes q^{\mathrm{op}})
=
\sum_{\alpha \colon w \to v}
\bigl(p\alpha\bigr) \,\otimes\,
I^{!}_{x}(\alpha^{\mathrm{op}})\bigl(q^{\mathrm{op}}\bigr).
\]
Suppose \(d^{\,n}_{v,w}(p \otimes q^{\mathrm{op}})=0\).
Because \(q^{\mathrm{op}}\neq 0\), there exists one arrow
\(\alpha^{\mathrm{op}}\colon v\to w\) in \(\Lambda^{!}\) for which
\(I^{!}_{x}(\alpha^{\mathrm{op}})\bigl(q^{\mathrm{op}}\bigr)\neq 0\).
For that \(\alpha\) we must then have \(p\alpha=0\) in \(\Lambda\).
By the monomial (path) nature of the relations, this forces a factorization
\(p=p'\beta\) with an arrow \(\beta\colon v\to u\) such that \(\beta\alpha=0\) in \(\Lambda\).

Consider
\(p' \otimes \bigl(q^{\mathrm{op}}\beta^{\mathrm{op}}\bigr)\in
\bigoplus_{u\in Q_{0}} P_{u}\langle n-1\rangle\otimes I^{!}_{x}(u)\).
A direct computation shows that
\[
d^{\,n+1}\bigl(p' \otimes q^{\mathrm{op}}\beta^{\mathrm{op}}\bigr)
= p \otimes q^{\mathrm{op}}.
\]
Thus \(\ker d^{\,n}\subseteq \mathrm{im}\, d^{\,n+1}\).
Since the argument applies to each elementary tensor in the chosen bases,
we conclude \(H^{n}\bigl(K(I^{!}_{x})\bigr)=0\) for all \(n\neq 0\).

To compute \(H^{0}\), note that the tail of the complex is
\[
\cdots \longrightarrow
 \bigoplus_{z \in Q_{0}} P_{z}\langle 1 \rangle \otimes I^{!}_{x}(z)
 \xrightarrow{d^{\,1}}
 P_{x} \otimes I^{!}_{x}(x)
 \longrightarrow 0 .
\]
Here \(I^{!}_{x}(x)\cong k\) is spanned by the trivial path at \(x\), and the image of
\(d^{\,1}\) is the graded radical of \(P_{x}\otimes I^{!}_{x}(x)\).
Hence the cokernel is canonically isomorphic to the simple \(S_{x}\) generated in
degree \(0\). Therefore \(H^{0}\bigl(K(I^{!}_{x})\bigr)\cong S_{x}\).

Finally, the differentials are homogeneous of internal degree \(0\) (after the displayed
shifts), and their matrices have entries in \(\Lambda_{1}\) (they are right multiplications
by arrows). Thus no degree–\(0\) isomorphism can occur in any differential, and the
resolution is minimal. This proves the claim.
\end{proof}
\medskip

We now recall the construction of the graded derived Koszul duality, starting from a cochain complex. Let
\[
\cdots \longrightarrow M^{p-1}
\xrightarrow{d^{p-1}} M^{p}
\xrightarrow{d^{p}} M^{p+1}
\longrightarrow \cdots
\]
be a bounded cochain complex in
\(\mathsf{C}^{b}\!\bigl(\Lambda^{!}\textup{-GMod}^{-,b}\bigr)\),
with each \(M^{p}\) right bounded in the internal grading, so that the complexes \(K(M^{p})\) have bounded cohomology.

For each \(p\), applying \(K\) to \(M^{p}\) yields a cochain complex
\(\bigl(K(M^{p})^{q},\, d_{2}^{q,p}\bigr)_{q\in\mathbb{Z}}\)
of graded projective \(\Lambda\)-modules with bounded cohomology. In this way, one obtains a double complex \(B^{\bullet,\bullet}\) with components
\[
B^{q,p}:=K(M^{p})^{q}
=\bigoplus_{u\in Q_{0}} P_{u}\langle q\rangle \otimes_{k} M^{p}_{q}(u),
\]
where \(M^{p}_{q}(u)\) denotes the homogeneous component of degree \(q\) at the vertex \(u\).

Adopting the cohomological convention in both directions, the horizontal differentials are given by
\[
d_{1}^{q,p} \colon B^{q,p}\longrightarrow B^{q,p+1}, \qquad
d_{1}^{q,p}
=
\bigoplus_{u\in Q_{0}}
\mathrm{id}_{P_{u}\langle q\rangle}\otimes (d^{p})_{q,u},
\]
and the vertical differentials by
\[
d_{2}^{q,p} \colon B^{q,p}\longrightarrow B^{q+1,p}.
\]
More precisely, the vertical differential is induced by the multiplication in the path algebra and is given on
\(P_{u}\langle q\rangle \otimes M^{p}_{q}(u)\) by
\[
d_{2}^{q,p}
=
\sum_{\alpha \colon u \to v}
P_{\alpha} \otimes M^{p}(\alpha^{\mathrm{op}}),
\]
where \(P_{\alpha}(p)=p\alpha\), and
\[
M^{p}(\alpha^{\mathrm{op}})\colon M^{p}_{q}(u)\longrightarrow M^{p}_{q+1}(v)
\]
is homogeneous of internal degree \(1\).

\medskip

The associated total complex is given by
\[
\mathrm{Tot}(B)^{n}
=
\bigoplus_{p+q=n} B^{q,p}, \qquad
d_{\mathrm{Tot}}^{n}
=
\sum_{p+q=n} \bigl(d_{1}^{q,p}+(-1)^{p}d_{2}^{q,p}\bigr),
\]
and one readily verifies that \(d_{\mathrm{Tot}}^{2}=0\).

\medskip

This construction defines a \(k\)-linear functor
\[
\mathcal{F}\colon
\mathsf{C}^{b}\!\bigl(\Lambda^{!}\textup{-Cop}^{\mathbb{Z}}\bigr)
\longrightarrow
\mathsf{C}^{-,b}\!\bigl(\Lambda\textup{-proj}^{\mathbb{Z}}\bigr).
\]
Since \(\mathcal{F}\) preserves homotopies, it descends to homotopy categories and induces a triangulated functor
\[
\mathscr{F}\colon
\mathsf{K}^{b}\!\bigl(\Lambda^{!}\textup{-Cop}^{\mathbb{Z}}\bigr)
\longrightarrow
\mathsf{K}^{-,b}\!\bigl(\Lambda\textup{-proj}^{\mathbb{Z}}\bigr).
\]

Moreover, by the Acyclic Assembly Lemma (see~\cite{24}), together with the canonical equivalence
\[
\mathsf{K}^{-,b}\!\bigl(\Lambda\textup{-proj}^{\mathbb{Z}}\bigr)
\;\cong\;
\mathsf{D}^{b}\!\bigl(\Lambda\textup{-gmod}\bigr),
\]
the functor \(\mathscr{F}\) preserves acyclic complexes. Consequently, it descends to a triangulated functor
\[
\mathfrak{F}\colon
\mathsf{D}^{b}\!\bigl(\Lambda^{!}\textup{-Cop}^{\mathbb{Z}}\bigr)
\longrightarrow
\mathsf{D}^{b}\!\bigl(\Lambda\textup{-gmod}\bigr).
\]

\medskip

The quasi-inverse functor
\[
\mathfrak{G}\colon
\mathsf{D}^{b}\!\bigl(\Lambda\textup{-gmod}\bigr)
\longrightarrow
\mathsf{D}^{b}\!\bigl(\Lambda^{!}\textup{-Cop}^{\mathbb{Z}}\bigr)
\]
is defined dually. To a finite-dimensional graded \(\Lambda\)-module \(N\), we associate a bounded cochain complex \(G(N)\) with components
\[
G(N)^{n}
=
\bigoplus_{x\in Q_{0}} I^{!}_{x}\langle n\rangle \otimes_{k} N_{n}(x),
\]
where \(I^{!}_{x}\) denotes the indecomposable graded injective \(\Lambda^{!}\)-module corresponding to the vertex \(x\).

The differentials
\[
d^{n}\colon G(N)^{n}\longrightarrow G(N)^{n+1}
\]
are defined componentwise by
\[
d^{n}
=
\sum_{\alpha\colon y\to x}
I^{!}_{\alpha^{\mathrm{op}}}\otimes N(\alpha),
\]
where the morphism
\[
I^{!}_{\alpha^{\mathrm{op}}}\colon I^{!}_{x}\longrightarrow I^{!}_{y}
\]
is induced by precomposition with \(\alpha^{\mathrm{op}}\). Explicitly, under the identification of \(I^{!}_{x}(u)\) with the space spanned by paths from \(u\) to \(x\), it is given by
\[
q \longmapsto
\begin{cases}
q' & \text{if } q = \alpha^{\mathrm{op}} q',\\
0  & \text{otherwise}.
\end{cases}
\]

We shall refer to the functor \(G\) as the \emph{coKoszul functor}.

\medskip 

This construction extends to complexes and yields a \(k\)-linear functor
\[
\mathcal{G}\colon
\mathsf{C}^{b}\!\bigl(\Lambda\textup{-gmod}\bigr)
\longrightarrow
\mathsf{C}^{b}\!\bigl(\Lambda^{!}\textup{-Inj}^{\mathbb{Z}}\bigr),
\]
whose image lies in the full subcategory
\(\mathcal{CLC}^{b}\!\bigl(\Lambda^{!}\textup{-Inj}^{\mathbb{Z}}\bigr)\)
of colinear complexes of injectives.

As in the case of \(\mathcal{F}\), the functor \(\mathcal{G}\) preserves homotopies. Consequently, it descends to homotopy categories and induces a triangulated functor
\[
\mathscr{G}\colon
\mathsf{K}^{b}\!\bigl(\Lambda\textup{-gmod}\bigr)
\longrightarrow
\mathsf{K}^{b}\!\bigl(\Lambda^{!}\textup{-Inj}^{\mathbb{Z}}\bigr).
\]

Moreover, since \(G\) sends acyclic complexes to acyclic complexes, the functor \(\mathscr{G}\) descends to derived categories, yielding a triangulated functor
\[
\mathfrak{G}\colon
\mathsf{D}^{b}\!\bigl(\Lambda\textup{-gmod}\bigr)
\longrightarrow
\mathsf{D}^{b}\!\bigl(\Lambda^{!}\textup{-Cop}^{\mathbb{Z}}\bigr).
\]

Finally, as in Theorem~2.3, the functor \(\mathfrak{G}\) sends projective modules \(P_x\) to the minimal injective coresolution of the simple module \(S^{!}_x\). The functors \(\mathfrak{F}\) and \(\mathfrak{G}\) are quasi-inverse equivalences.

\medskip

Let \(M=\bigoplus_{i\in\mathbb Z} M_i\) be a locally finite-dimensional right bounded graded
\(\Lambda^{!}\)-module, and let \(r\in\mathbb Z\). The \emph{truncation of \(M\) at degree \(r\)} is the graded module
\[
\tau_{\le r}M
=
\bigoplus_{i\le r} M_i,
\]
whose graded components are given by
\[
(\tau_{\le r}M)_i =
\begin{cases}
M_i & \text{if } i\le r,\\
0   & \text{if } i>r.
\end{cases}
\]
We also set
\[
\tau_{> r}M
=
\bigoplus_{i> r} M_i.
\]
Since \(M\) is right bounded and locally finite-dimensional, the module \(\tau_{> r}M\) is finite-dimensional, while \(\tau_{\le r}M\) is again right bounded. Moreover, these modules fit into a short exact sequence in \(\Lambda^{!}\textup{-GMod}^{-}\):
\[
0 \longrightarrow \tau_{> r}M \longrightarrow M \longrightarrow \tau_{\le r}M \longrightarrow 0.
\]

\medskip

Dually, let \(N=\bigoplus_{i\in\mathbb Z} N_i\) be a locally finite-dimensional left bounded graded
\(\Lambda\)-module, and let \(r\in\mathbb Z\). The \emph{truncation of \(N\) at degree \(r\)} is the graded module
\[
\tau_{\ge r}N
=
\bigoplus_{i\ge r} N_i,
\]
whose graded components are given by
\[
(\tau_{\ge r}N)_i =
\begin{cases}
N_i & \text{if } i\ge r,\\
0   & \text{if } i<r.
\end{cases}
\]
We also set
\[
\tau_{< r}N
=
\bigoplus_{i< r} N_i.
\]
Then \(\tau_{< r}N\) is finite-dimensional, while \(\tau_{\ge r}N\) is again left bounded, and we obtain a short exact sequence in \(\Lambda\textup{-GMod}^{+}\):
\[
0 \longrightarrow \tau_{< r}N \longrightarrow N \longrightarrow \tau_{\ge r}N \longrightarrow 0.
\]

\medskip

Let \(M \in \Lambda\textup{-GMod}^{+}\) (respectively, \(M \in \Lambda^{!}\textup{-GMod}^{-}\)). We say that \(M\) admits a \emph{linear truncation} (respectively, a \emph{colinear truncation}) if there exists an integer \(r\in\mathbb Z\) such that
\[
\tau_{\ge r}M\langle r\rangle
\qquad
(\textup{respectively, }\tau_{\le r}M\langle -r\rangle)
\]
is Koszul (respectively, coKoszul).

\medskip 

We next recall the notion of absolutely Koszul algebras, introduced by Iyengar and R\"omer in the commutative setting~\cite{19}. We show that noncommutative quadratic monomial algebras are absolutely Koszul and, moreover, have global linearity defect at most two. As a consequence, every finitely presented module over such an algebra has a rational Poincar\'e series.

\medskip

We also note that commutative quadratic monomial algebras are not absolutely Koszul in general; see, for example,~\cite{28}. It remains an open problem to classify commutative quadratic monomial algebras with respect to the absolutely Koszul property; see~\cite[Question~4.15]{8}.

\medskip

Let \( M \) be a finitely presented graded \( \Lambda \)-module, and let
\[
\cdots \longrightarrow P^{-k} \longrightarrow \cdots \longrightarrow P^{-1} \longrightarrow P^{0} \longrightarrow M \longrightarrow 0
\]
be a minimal graded projective resolution of \( M \).

Since \( \Lambda \) is graded, each differential \( d^{i} \colon P^{i} \to P^{i+1} \) decomposes into homogeneous components with respect to the internal grading,
\[
d^{i} = \sum_{j \ge 1} d^{i}_{j},
\]
where \( d^{i}_{j} \) is homogeneous of degree \( j \).

The \emph{linear part} of \( P^{\bullet} \), denoted \( \mathrm{Lin}^{\Lambda}(P^{\bullet}) \), is the complex with the same underlying graded modules as \( P^{\bullet} \), and whose differentials are given by the degree-one components
\(d^{i}_{\mathrm{lin}} := d^{i}_{1}.\)
Equivalently, \( \mathrm{Lin}^{\Lambda}(P^{\bullet}) \) is obtained from \( P^{\bullet} \) by discarding all homogeneous components of the differentials of degree greater than one.

Since \(\Lambda\) is given by a quiver with relations, each homogeneous component \(d^{i}_{j}\) is induced by multiplication by paths of length \(j\). In particular, the degree-one components \(d^{i}_{1}\) are induced by multiplication by arrows in the quiver of \(\Lambda\). It follows that \(\mathrm{Lin}^{\Lambda}(P^{\bullet})\) is a complex of graded projective modules.

The notion of the linear part was introduced by Eisenbud, Fl{\o}ystad, and Schreyer in the context of free resolutions over exterior algebras; see~\cite{10}.

\medskip

The \emph{linearity defect} of \(M\) is defined to be 
\[
\mathrm{ld}_{\Lambda}(M)
=
\sup\Bigl\{
i 
\;\Big|\;
H^{-i}\bigl(\mathrm{Lin}^{\Lambda}(P^{\bullet})\bigr)\neq 0
\Bigr\}.
\]

\medskip

The \emph{global linearity defect} of \(\Lambda\) is defined by
\[
\mathrm{gl\,ld}_{\Lambda}
=
\sup\Bigl\{
\mathrm{ld}_{\Lambda}(M)
\;\Big|\;
M\in \Lambda\textup{-GMod} \text{ is finitely presented}
\Bigr\}.
\]

\medskip

A graded \(\Lambda\)-module \(M\) is said to be \emph{weakly Koszul} if the linear part of its minimal projective resolution is exact. In particular,
\[
\mathrm{ld}_{\Lambda}(M)=0
\quad \text{if and only if \(M\) is weakly Koszul}.
\]

Weakly Koszul modules were termed \emph{Koszul} in~\cite{16}. However, since we work in the noncommutative setting, we follow the terminology in~\cite{24}.
\section{Main Results}

In this section, we establish the main results on Koszul duality stated in the introduction. These results show that quadratic monomial algebras exhibit particularly well-behaved duality properties, despite the fact that their Koszul duals are infinite-dimensional.

\subsection{Graded modules over quadratic monomial algebras}

The main objects of this paper are certain pairs of distinguished families
of left \(\Lambda\)-modules of the form \(\Lambda \alpha\), together with
corresponding submodules of indecomposable injective left
\(\Lambda^{!}\)-modules naturally associated with the arrows of the quiver
of \(\Lambda^{!}\).

More precisely, for each vertex \(x \in Q_{0}\), let \(I^{!}_{x}\) denote
the indecomposable injective left \(\Lambda^{!}\)-module corresponding to \(x\).
Given an arrow \(\alpha^{\mathrm{op}} \colon y \to x\) in the quiver of
\(\Lambda^{!}\), we consider the submodule
\(\alpha^{\mathrm{op}} I^{!}_{x} \;\subseteq\; I^{!}_{x}\)
generated by the image of \(\alpha^{\mathrm{op}}\).

Equivalently, this submodule consists of all elements of \(I^{!}_{x}\)
represented by paths in the opposite quiver that factor through
\(\alpha^{\mathrm{op}}\). That is, if \(z \in Q_{0}\) and \(q^{\mathrm{op}}\)
is a nonzero path in the opposite quiver from \(z\) to \(x\), then
\[
q^{\mathrm{op}} \in \alpha^{\mathrm{op}} I^{!}_{x}
\quad \Longleftrightarrow \quad
q^{\mathrm{op}} = \alpha^{\mathrm{op}} p^{\mathrm{op}}
\]
for some path \(p^{\mathrm{op}}\) in \(\Lambda^{!}\) with source
\(s(p^{\mathrm{op}})=z\).

In this way, we obtain a family of submodules
\[
\bigl\{\, \alpha^{\mathrm{op}} I^{!}_{x}
\;\big|\;
\alpha^{\mathrm{op}} \colon y \to x,\ x \in Q_{0}
\,\bigr\}.
\]

The following proposition plays a central role in this paper. Its proof is
similar to that of Theorem~2.3.

\begin{proposition}
The Koszul functor
\[
K \colon \Lambda^{!}\textup{-GMod} \longrightarrow
\mathcal{LC}\!\bigl(\Lambda\textup{-proj}^{\mathbb{Z}}\bigr)
\]
sends the module \(\alpha^{\mathrm{op}} I_x^!\) to a minimal projective resolution of the \(\Lambda\)-module \(\Lambda \alpha\). In particular, \(\Lambda \alpha\) is linear.

Dually, the coKoszul functor
\[
G \colon \Lambda\textup{-GMod} \longrightarrow
\mathcal{CLC}\!\bigl(\Lambda^{!}\textup{-Inj}^{\mathbb{Z}}\bigr)
\]
sends the module \(\Lambda \alpha\) to a minimal injective coresolution of the \(\Lambda^!\)-module \(\alpha^{\mathrm{op}} I_x^!\). In particular, \(\alpha^{\mathrm{op}} I_x^!\) is colinear.
\end{proposition}

\begin{proof}
We prove only the first statement, the second being dual. The argument is entirely analogous to that of Theorem~2.3, with \(I_x^!\) replaced by the submodule \(\alpha^{\mathrm{op}} I_x^!\) and \(S_x\) replaced by \(\Lambda \alpha\).

The cochain complex \(K(\alpha^{\mathrm{op}} I_x^!)\) has terms
{\scriptsize
\[
\cdots \longrightarrow
 \bigoplus_{a \in Q_{0}} P_{a}\langle n-1 \rangle \otimes (\alpha^{\mathrm{op}} I_x^!)_{n-1}(a)
 \longrightarrow
 \bigoplus_{b \in Q_{0}} P_{b}\langle n \rangle \otimes (\alpha^{\mathrm{op}} I_x^!)_{n}(b)
 \longrightarrow
 \bigoplus_{c \in Q_{0}} P_{c}\langle n+1 \rangle \otimes (\alpha^{\mathrm{op}} I_x^!)_{n+1}(c)
 \longrightarrow \cdots .
\]
}

Since \(\Lambda\) is quadratic monomial, the nonzero paths form \(k\)-bases of the graded components of the projective modules \(P_a\), and the nonzero opposite paths factoring through \(\alpha^{\mathrm{op}}\) form \(k\)-bases of the graded components of \(\alpha^{\mathrm{op}} I_x^!\). The same combinatorial argument as in Theorem~2.3 shows that this complex is exact in every nonzero degree.

It remains to compute the cohomology in degree \(1\). Consider the tail
\[
\cdots \longrightarrow
 \bigoplus_{z \in Q_{0}} P_{z}\langle 2 \rangle \otimes (\alpha^{\mathrm{op}} I_x^!)_{2}(z)
 \xrightarrow{d^{2}}
 P_{y}\langle 1 \rangle \otimes (\alpha^{\mathrm{op}} I_x^!)_{1}(y)
 \longrightarrow 0 .
\]
Since \((\alpha^{\mathrm{op}} I_x^!)_{1}(y)\cong k\), this identifies with
\[
\cdots \longrightarrow
 \bigoplus_{z \in Q_{0}} P_{z}\langle 2 \rangle \otimes (\alpha^{\mathrm{op}} I_x^!)_{2}(z)
 \xrightarrow{d^{2}}
 P_{y}\langle 1 \rangle 
 \longrightarrow 0 .
\]

By construction, the image of \(d^{2}\) is
\[
\bigoplus_{\beta \colon y \to z,\ \beta\alpha=0} \Lambda \beta\langle 1 \rangle ,
\]
that is, the sum of all paths \(\beta\) such that \(\beta\alpha=0\) in \(\Lambda\). This coincides with the kernel of the canonical morphism
\[
P_{y}\langle 1 \rangle \longrightarrow \Lambda \alpha,
\qquad e_{y} \longmapsto \alpha.
\]
Thus the cokernel is isomorphic to the left ideal \(\Lambda \alpha\), and hence
\[
H^{1}\bigl(K(\alpha^{\mathrm{op}} I_x^!)\bigr)\cong \Lambda\alpha,
\qquad
H^{n}\bigl(K(\alpha^{\mathrm{op}} I_x^!)\bigr)=0 \ \text{for } n\neq 1.
\]
\end{proof}

We shall require the following lemma.

\begin{lemma}
Let
\[
P^{1} \xrightarrow{d} P^{0} \longrightarrow M \longrightarrow 0
\]
be a minimal projective presentation of a finitely presented graded \(\Lambda\)-module \(M\), where each projective is a finite direct sum of indecomposable modules of the form \(P_x\langle n\rangle\). Then
\[
\ker d \;=\;
\bigoplus_{\substack{\alpha \colon x \to y }}
\Lambda \alpha\langle i \rangle ,
\]
\end{lemma}

\begin{proof}
Write
\[
P^{1} = \bigoplus_{y \in Q_0,\; n \in \mathbb{Z}} P_y\langle n\rangle,
\qquad
P^{0} = \bigoplus_{x \in Q_0,\; m \in \mathbb{Z}} P_x\langle m\rangle.
\]
The differential \(d \colon P^{-1} \to P^{0}\) is given componentwise by morphisms
\[
P_y\langle n\rangle \longrightarrow P_x\langle m\rangle,
\]
that is, by finite sums of the form
\[
\sum_{q \colon x \to y} P_q,
\]
where \(P_q\) denotes right multiplication by the path \(q\) (and arrows correspond to paths of length one).

Let \(p \in P_y\langle n\rangle\) be a homogeneous element such that \(d(p)=0\). Since \(\Lambda\) is a quadratic monomial algebra, it follows that for every path \(q\) from \(x\) to \(y\) occurring in the expression of \(d\), one has
\[
p q = 0 \quad \text{in } \Lambda.
\]
Thus there exist an arrow \(\alpha \colon y \to z\) and a path \(p'\) such that
\[
p = p' \alpha,
\qquad
\alpha q = 0.
\]

It follows that the kernel is a finite direct sum of principal ideals generated by arrows \(\alpha \colon y \to z\), that is,
\[
\ker d \;=\;
\bigoplus_{\substack{\alpha \colon y \to z \\ i \in \mathbb{Z} }}
\Lambda \alpha\langle i\rangle.
\]
\end{proof}
We are now ready to prove the first main result of this subsection.

\begin{theorem}
The algebra \(\Lambda\) is  left graded coherent (respectively, left graded cocoherent). 
\end{theorem}

\begin{proof}
Let \(\Lambda\) be a quadratic monomial algebra. To prove coherence, it suffices to show that the kernel of any morphism between finitely generated projective \(\Lambda\)-modules is finitely presented.

Let \(f \colon P \to Q\) be a morphism between finitely generated projective \(\Lambda\)-modules. By Lemma~3.2, the kernel \(\ker(f)\) decomposes as a finite direct sum of cyclic submodules of the form \(\Lambda \alpha\).

By proposition 3.1, each module \(\Lambda \alpha\) admits a linear resolution, and in particular a linear presentation. It follows that each \(\Lambda \alpha\) is finitely presented. Since a finite direct sum of finitely presented modules is finitely presented, we conclude that \(\ker(f)\) is finitely presented.

Thus \(\Lambda\) is coherent. The dual argument yields cocoherence.
\end{proof}
The second main result shows that noncommutative quadratic monomial algebras exhibit favorable homological properties, as do their modules.

\begin{theorem}
The algebra $\Lambda$ is absolutely Koszul and satisfies
\(\mathrm{gl\,ld}_{\Lambda} \le 2.\)
\end{theorem}

\begin{proof}
By Lemma~3.2, the second syzygy of any finitely presented graded
\(\Lambda\)-module is weakly Koszul. It follows that every such module
has linear defect at most two. In particular, \(\Lambda\) is absolutely
Koszul and satisfies \(\mathrm{gl\,ld}_{\Lambda} \le 2\).
\end{proof}
\begin{corollary}
If \(\Lambda\) has radical square zero, then
\(\mathrm{gl\,ld}_{\Lambda} = 0.\)
In particular, every finitely presented graded \(\Lambda\)-module is weakly Koszul.
\end{corollary}

\begin{proof}
If \(\operatorname{rad}^{2}\Lambda=0\), then every path of length at least two vanishes. Hence, in any minimal graded projective resolution of a finitely presented \(\Lambda\)-module, all differentials are induced by multiplication by arrows, and therefore are homogeneous of degree one. Thus the linear part of the resolution coincides with the resolution itself, so it is exact. It follows that every finitely presented \(\Lambda\)-module is weakly Koszul, and consequently \(\mathrm{gl\,ld}_{\Lambda}=0\).
\end{proof}
Recall that if \(M\) is a finitely presented graded \(\Lambda\)-module, its
Poincar\'e series is defined by
\[
P^{\Lambda}_{M}(t)
=
\sum_{n \ge 0}
\dim_{k}\!\bigl(\operatorname{Ext}^{n}_{\Lambda}(M,\Lambda_{0})\bigr)\, t^{n}.
\]
As a direct consequence of Theorem~3.4, we obtain the following.

\begin{corollary}
Every finitely presented graded $\Lambda$-module has a rational Poincar\'e series.
\end{corollary}

\begin{proof}
This follows from Theorem~3.4 and \cite[Theorem~4.7]{24}.
\end{proof}

For the remainder of the paper, we assume that \(\Lambda\) is finite-dimensional. Consequently, its Koszul dual \(\Lambda^{!}\) has finite global dimension.
\subsection{Koszul duality}
We now use the preceding results to refine our Koszul dualities and to
study the categories of perfect and coperfect modules, as well as the
tails category of \(\Lambda^{!}\) and its dual.

We begin with the following proposition, which is a direct consequence of
Theorem~3.3.

\begin{proposition}
Let \(\Lambda\) be a finite-dimensional quadratic monomial algebra. Then
\(\Lambda^{!}\textup{-Pe}^{\mathbb{Z}}\) (resp. \(\Lambda^{!}\textup{-Cop}^{\mathbb{Z}}\)) is an abelian category and coincides with the category of finitely presented (resp. finitely copresented) graded \(\Lambda^{!}\)-modules, that is,
\[
\Lambda^{!}\textup{-Pe}^{\mathbb{Z}} = \Lambda^{!}\textup{-Fp}^{\mathbb{Z}},
\qquad
\Lambda^{!}\textup{-Cop}^{\mathbb{Z}} = \Lambda^{!}\textup{-Fcp}^{\mathbb{Z}}.
\]
In particular, the tails and cotails categories
\[
\Lambda^{!}\textup{-Fp}^{\mathbb{Z}} \big/ \Lambda^{!}\textup{-gmod}
\quad \text{and} \quad
\Lambda^{!}\textup{-Fcp}^{\mathbb{Z}} \big/ \Lambda^{!}\textup{-gmod}
\]
are abelian.
\end{proposition}

We now state the graded derived Koszul duality and its consequences.

\begin{theorem}
Let \(\Lambda\) be a finite-dimensional quadratic monomial algebra. Then the following statements hold.

\begin{enumerate}
\item \textbf{Graded derived Koszul duality.}
There are triangulated equivalences
\[
\mathsf{D}^{b}\bigl(\Lambda^{!}\textup{-Fcp}^{\mathbb{Z}}\bigr)
\;\xrightarrow{\ \sim\ }\;
\mathsf{D}^{b}\bigl(\Lambda\textup{-gmod}\bigr),
\qquad
\mathsf{D}^{b}\bigl(\Lambda\textup{-gmod}\bigr)
\;\xrightarrow{\ \sim\ }\;
\mathsf{D}^{b}\bigl(\Lambda^{!}\textup{-Fp}^{\mathbb{Z}}\bigr).
\]

\item \textbf{Graded singular Koszul duality.}
There is a triangulated equivalence
\[
\mathsf{D}^{b}\!\Bigl(
\Lambda^{!}\textup{-Fcp}^{\mathbb{Z}} \big/ \Lambda^{!}\textup{-gmod}
\Bigr)
\;\xrightarrow{\ \sim\ }\;
\mathsf{D}_{\mathrm{sg}}\bigl(\Lambda\textup{-gmod}\bigr).
\]

\item \textbf{The graded BGG correspondance.}
If, in addition, \(\Lambda\) is Iwanaga--Gorenstein, then there is a triangulated equivalence
\[
\Lambda\textup{-}\underline{\mathrm{Gproj}}^{\mathbb{Z}}
\;\xrightarrow{\ \sim\ }\;
\mathsf{D}^{b}\!\Bigl(
\Lambda^{!}\textup{-Fp}^{\mathbb{Z}} \big/ \Lambda^{!}\textup{-gmod}
\Bigr)   \;\xrightarrow{\ \sim\ }\;\mathsf{D}^{b}\!\Bigl(
\Lambda^{!}\textup{-Fcp}^{\mathbb{Z}} \big/ \Lambda^{!}\textup{-gmod}
\Bigr)
\]
\end{enumerate}
\end{theorem}
As a direct consequence of Theorem~3.8(2), we obtain the following.

\begin{corollary}
The categories
\[
\Lambda^{!}\textup{-Fcp}^{\mathbb{Z}} \big/ \Lambda^{!}\textup{-gmod}
\quad \textup{and} \quad
\Lambda^{!}\textup{-Fp}^{\mathbb{Z}} \big/ \Lambda^{!}\textup{-gmod}
\]
are hereditary.
\end{corollary}

\begin{proof}
It is well known that any object of the singularity category
\(\mathsf{D}_{\mathrm{sg}}\bigl(\Lambda\textup{-gmod}\bigr)\)
is isomorphic to a shift of a finite-dimensional graded \(\Lambda\)-module.
Moreover, since \(\Lambda\) is quadratic monomial, every such module can be written as a finite direct sum of modules
of the form \(\Lambda \beta\langle i\rangle\).

On the other hand, the equivalence \(\mathfrak{G}\) sends
\(\bigoplus \Lambda \beta\langle i\rangle\) to a bounded complex of injective
\(\Lambda^{!}\)-modules. In particular, by Proposition~3.1, in the category
\[
\mathsf{D}^{b}\!\Bigl(
\Lambda^{!}\textup{-Fcp}^{\mathbb{Z}} \big/ \Lambda^{!}\textup{-gmod}
\Bigr),
\]
this complex is isomorphic to its cohomology.

It follows that every object is isomorphic to a direct sum of shifts of
objects concentrated in a single degree. Consequently, the category
\(\Lambda^{!}\textup{-Fcp}^{\mathbb{Z}} \big/ \Lambda^{!}\textup{-gmod}\)
is hereditary.

For the category
\(\Lambda^{!}\textup{-Fp}^{\mathbb{Z}} \big/ \Lambda^{!}\textup{-gmod},\)
one proceeds similarly by working with the \(\Lambda\)-modules of the form
\(\alpha I_{x}\) and the Verdier quotient
\[
\mathsf{D}^{b}\bigl(\Lambda\textup{-gmod}\bigr)
\big/ \mathsf{K}^{b}\bigl(\mathrm{inj}^{\mathbb{Z}}\bigr),
\]
together with the corresponding bounded derived category
\[
\mathsf{D}^{b}\!\Bigl(
\Lambda^{!}\textup{-Fp}^{\mathbb{Z}} \big/ \Lambda^{!}\textup{-gmod}
\Bigr).
\]
The same argument shows that this category is hereditary.
\end{proof}

In what follows, we establish new characterizations of finitely presented and finitely copresented $\Lambda^{!}$-modules. As an application, we obtain explicit descriptions of the objects in the corresponding tails and cotails categories. We begin with the following well-known result describing the cohomology of objects in the bounded derived category.

\begin{lemma}
Let $X^\bullet$ be an object of $\mathsf{D}^b(\Lambda\textup{-gmod})$. Then, for each $n \in \mathbb{Z}$, there is a natural isomorphism
\[
H^n(X^\bullet) \cong \bigoplus_{i \in \mathbb{Z}} \operatorname{Hom}_{\mathsf{D}^b(\Lambda\textup{-gmod})}\bigl(\Lambda, X^\bullet[n]\langle i \rangle\bigr).
\]
\end{lemma}
We shall also need the following lemma; compare with~\cite[Lemma~4.2]{5}. Note that the grading shift used in~\cite{5} is opposite to the convention adopted in the present paper.

\begin{lemma}
Let \( X^\bullet \in \mathsf{D}^{b}(\Lambda^{!}\textup{-Fcp}^{\mathbb{Z}}) \). Then there is a natural isomorphism
\[
\mathfrak{F}\bigl(X^\bullet\langle i \rangle\bigr)
\;\cong\;
\mathfrak{F}(X^\bullet)\langle i \rangle[-i].
\]
\end{lemma}
The following proposition shows that the categories of finitely presented and finitely copresented modules can be characterized by the existence of linear and colinear truncations, respectively.

\begin{proposition}
Let \(M\) be a finitely presented graded \(\Lambda\)-module. Then there exists a short exact sequence
\[
0 \longrightarrow N \longrightarrow M \longrightarrow L \longrightarrow 0,
\]
where \(N\) is linear and \(L\) is finite-dimensional.

Dually, if \(M\) is a finitely copresented graded \(\Lambda^{!}\)-module, then there exists a short exact sequence
\[
0 \longrightarrow N' \longrightarrow M \longrightarrow L' \longrightarrow 0,
\]
where \(N'\) is finite-dimensional and \(L'\) is colinear.
\end{proposition}

\begin{proof}
We prove the dual statement; the first follows by duality. Let $M$ be a finitely copresented $\Lambda^{!}$-module. By Proposition~3.4 in~\cite{5}, the linear complex $K(M)$ has bounded cohomology.

Let $n$ be the largest integer such that
\(H^{n+1}\bigl(K(M)\bigr)\ne 0.\)
Consider the short exact sequence
\[
0 \longrightarrow \tau_{> n}M \longrightarrow M \longrightarrow \tau_{\le n}M \longrightarrow 0.
\]
We claim that $\tau_{\le n}M$ is colinear, equivalently that $\tau_{\le n}M\langle -n\rangle$ is coKoszul.

Since truncation only removes components in degrees $> n$, it follows that the complex $K(\tau_{\le n}M\langle -n\rangle)$ has vanishing cohomology in all degrees except possibly in degree $0$, that is,
\[
H^{j}\bigl(K(\tau_{\le n}M\langle -n\rangle)\bigr)=0
\quad \text{for all } j \neq 0.
\]

By Theorem~3.8(1), for all $x \in Q_0$ and $i \in \mathbb{Z}$, there is an isomorphism
\[
\begin{aligned}
\operatorname{Ext}^k_{\Lambda^{!}\textup{-Fcp}^{\mathbb{Z}}}
\bigl(S^{!}_x\langle -i\rangle, \tau_{\le n}M\langle -n\rangle\bigr)
&\cong
\operatorname{Hom}_{\mathsf{D}^b(\Lambda\textup{-gmod})}
\bigl(P_x\langle -i\rangle[i], K(\tau_{\le n}M\langle -n\rangle)[k]\bigr) \\
&\cong
\operatorname{Hom}_{\mathsf{D}^b(\Lambda\textup{-gmod})}
\bigl(P_x\langle -i\rangle, K(\tau_{\le n}M\langle-n\rangle)[k-i]\bigr).
\end{aligned}
\]
The vanishing of $H^{l}\bigl(K(\tau_{\le n}M\langle -n\rangle)\bigr)$ for $l \neq 0$ implies that the right-hand side vanishes unless $k-i=0$, that is, $i=k$. Hence
\[
\operatorname{Ext}^k_{\Lambda^{!}\textup{-Fcp}^{\mathbb{Z}}}
\bigl(S^{!}_x\langle -i\rangle, \tau_{\le n}M\langle -n\rangle\bigr)=0
\quad \text{for all } i \neq k,
\]
for all $x \in Q_0$. This shows that $\tau_{\le n}M\langle -n\rangle$ is coKoszul, and hence $\tau_{\le n}M$ is colinear.
\end{proof}

We obtain a consequence concerning the Hilbert series of finitely copresented and finitely presented modules over \(\Lambda^{!}\).

First, recall that the Hilbert series of a locally finite-dimensional graded module 
\(M \in \Lambda^{!}\textup{-GMod}\) is defined by
\[
H_M(t)
=
\sum_{i \in \mathbb{Z}} \bigl(\dim_k M_i\bigr)\, t^i.
\]

\begin{corollary}
Let \(M \in \Lambda^{!}\textup{-GMod}\) be a finitely copresented (resp.\ finitely presented) graded module. Then the Hilbert series \(H_M(t)\) is a rational function.
\end{corollary}

\begin{proof}
By Proposition~3.12, finitely presented \(\Lambda^{!}\)-modules are precisely those satisfying the property~(L) in~\cite[Definition~5.1]{27}. Hence the rationality of the Hilbert series follows from~\cite[Corollary~6.6]{27}.

The corresponding statement for finitely copresented modules is obtained by a dual argument.
\end{proof}
The following corollary provides a characterization of the tails and cotails categories
\[
\Lambda^{!}\textup{-Fcp}^{\mathbb{Z}} \big/ \Lambda^{!}\textup{-gmod}
\quad \textup{and} \quad
\Lambda^{!}\textup{-Fp}^{\mathbb{Z}} \big/ \Lambda^{!}\textup{-gmod}.
\]

\begin{corollary}
Every object in the category \(\Lambda^{!}\textup{-Fp}^{\mathbb{Z}} \big/ \Lambda^{!}\textup{-gmod}\) (resp.\ \(\Lambda^{!}\textup{-Fcp}^{\mathbb{Z}} \big/ \Lambda^{!}\textup{-gmod}\)) is isomorphic to a linear (resp.\ colinear) \(\Lambda^{!}\)-module.
\end{corollary}

\begin{proof}
This follows immediately from Proposition~3.12.
\end{proof}
The ungraded derived and singular Koszul dualities admit the following refinement.

\begin{theorem}
Let \(\Lambda\) be a finite-dimensional quadratic monomial algebra. Then the following statements hold.

\begin{enumerate}
\item \textbf{Ungraded derived Koszul duality.}
There is a triangulated equivalence
\[
\mathrm{H}^{0}\!\Bigl(
\operatorname{pretr}\bigl(
\mathsf{D}^{b}_{\mathrm{dg}}\bigl(\Lambda^{!}\textup{-Fcp}^{\mathbb{Z}}\bigr)
\big/ \langle 1 \rangle[1]
\bigr)
\Bigr)
\;\xrightarrow{\ \sim\ }\;
\mathsf{D}^{b}\bigl(\Lambda\textup{-mod}\bigr).
\]

\item \textbf{Ungraded singular Koszul duality.}
There is a triangulated equivalence
\[
\mathsf{D}^{b}\bigl(
\Lambda^{!}\textup{-Fcp}^{\mathbb{Z}} \big/ \Lambda^{!}\textup{-gmod}
\bigr)
\big/ \langle 1 \rangle[1]
\;\xrightarrow{\ \sim\ }\;
\mathsf{D}_{\mathrm{sg}}\bigl(\Lambda\textup{-mod}\bigr).
\]
\item \textbf{Ungraded BGG correspondance.}. If, in addition, \(\Lambda\) is Iwanaga--Gorenstein, then there is a triangulated equivalence
\[\mathsf{D}^{b}\bigl(
\Lambda^{!}\textup{-Fcp}^{\mathbb{Z}} \big/ \Lambda^{!}\textup{-gmod}
\bigr)
\big/ \langle 1 \rangle[1]
\;\xrightarrow{\ \sim\ }\;
\Lambda\textup{-}\underline{\mathrm{Gproj}}\]
\end{enumerate}
\end{theorem}

\begin{proof}
For the second statement, it suffices to show that the forgetful functor
\[
F \colon \mathsf{D}_{\mathrm{sg}}\bigl(\Lambda\textup{-gmod}\bigr)
\longrightarrow
\mathsf{D}_{\mathrm{sg}}\bigl(\Lambda\textup{-mod}\bigr)
\]
is dense; see \cite[Section~4]{5}.

Let \(M\) be a finite-dimensional \(\Lambda\)-module. Every syzygy of \(M\) decomposes as a finite
direct sum of modules of the form \(\Lambda \alpha\), and is therefore graded.

It follows that every object of
\(\mathsf{D}_{\mathrm{sg}}\bigl(\Lambda\textup{-mod}\bigr)\) is isomorphic to
the image under \(F\) of an object in
\(\mathsf{D}_{\mathrm{sg}}\bigl(\Lambda\textup{-gmod}\bigr)\). Hence \(F\) is
dense, and the claimed equivalence follows.
\end{proof}
We shall now use the semisimplicity of the stable category of
Gorenstein-projective modules; see~\cite[Theorem~5.7]{7}. This yields the
following structural consequence for the tail and cotail categories associated
with \(\Lambda^{!}\).

\begin{corollary}
Assume that \(\Lambda\) is Iwanaga--Gorenstein. Then the quotient categories
\[
\Lambda^{!}\textup{-Fcp}^{\mathbb{Z}}
\big/
\Lambda^{!}\textup{-gmod}
\qquad\text{and}\qquad
\Lambda^{!}\textup{-Fp}^{\mathbb{Z}}
\big/
\Lambda^{!}\textup{-gmod}
\]
are semisimple.
\end{corollary}

\begin{proof}
By Theorem~3.14, there is an equivalence
\[
\mathsf{D}^{b}\Bigl(
\Lambda^{!}\textup{-Fcp}^{\mathbb{Z}} \big/ \Lambda^{!}\textup{-gmod}
\Bigr)
\big/ \langle 1 \rangle[1]
\;\xrightarrow{\ \sim\ }\;
\Lambda\textup{-}\underline{\mathrm{Gproj}},
\]
which arises from a \(\langle 1 \rangle[1]\)-Galois covering. Since such a covering is faithful (see~\cite[Remark~2.33]{5}), and every exact triangle in \(\Lambda\textup{-}\underline{\mathrm{Gproj}}\) splits; see~\cite[Theorem~5.7]{7}. It follows that every triangle in
\[
\mathsf{D}^{b}\Bigl(
\Lambda^{!}\textup{-Fcp}^{\mathbb{Z}} \big/ \Lambda^{!}\textup{-gmod}
\Bigr)
\]
splits as well. Therefore, the category
\(\Lambda^{!}\textup{-Fcp}^{\mathbb{Z}} \big/ \Lambda^{!}\textup{-gmod}\)
is semisimple. The statement for
\(\Lambda^{!}\textup{-Fp}^{\mathbb{Z}} \big/ \Lambda^{!}\textup{-gmod}\)
follows dually.
\end{proof}
We recall the following standard fact. Let \(\mathcal A\) be an abelian
semisimple category and let \(T\colon \mathcal A \to \mathcal A\) be an
autoequivalence. Then \(\mathcal A\) carries a split triangulated structure
with translation functor \(T\). The distinguished triangles are exactly those
triangles which, up to isomorphism, decompose as finite direct sums of the
elementary split triangles
\[
X \longrightarrow 0 \longrightarrow T X \xrightarrow{\mathrm{id}} T X,
\qquad
X \xrightarrow{\mathrm{id}} X \longrightarrow 0 \longrightarrow T X,
\qquad
0 \longrightarrow X \xrightarrow{\mathrm{id}} X \longrightarrow 0.
\]

\begin{lemma} Assume that \(\Lambda\) is Iwanaga--Gorenstein. Then
the quotient category
\(\Lambda^{!}\textup{-Fcp}^{\mathbb Z}/\Lambda^{!}\textup{-gmod},\)
equipped with the grading shift functor \(\langle 1\rangle\), admits a split
triangulated structure.
\end{lemma}
Recall that there is a functor
\[
K \colon 
\Lambda^!\textup{-GMod} 
   \longrightarrow 
   \mathsf{C}\bigl(\Lambda\textup{-proj}^{\mathbb{Z}}\bigr),
\]
whose restriction to the category of finitely copresented modules is of the form
\[
K \colon 
\Lambda^!\textup{-Fcp}^{\mathbb{Z}}
   \longrightarrow
   \mathsf{C}^{-,b}\bigl(\Lambda\textup{-proj}^{\mathbb{Z}}\bigr).
\]
By the construction of \(K\), this functor sends the subcategory
\(\Lambda^!\textup{-gmod}\) to objects which become zero in the singularity
category. Hence it induces an additive functor
\(\overline{K}\) fitting into the following commutative diagram:
\[
\begin{tikzcd}[column sep=small]
\Lambda^!\textup{-Fcp}^{\mathbb{Z}}
   \ar[d, two heads, "\pi"'] 
   \ar[r, "K"] 
 & \mathsf{C}^{-,b}\bigl(\Lambda^{\mathbb{Z}}\textup{-proj}\bigr) 
   \ar[r] 
 & \mathsf{K}^{-,b}\bigl(\Lambda^{\mathbb{Z}}\textup{-proj}\bigr)
   \ar[r, "\sim"] 
 & \mathsf{D}^b(\Lambda\textup{-gmod}) 
   \ar[r] 
 & \mathsf{D}_{\mathrm{sg}}(\Lambda\textup{-gmod}) 
\\
\Lambda^!\textup{-Fcp}^{\mathbb{Z}} / \Lambda^!\textup{-gmod}
   \ar[rrrru, dashed, "\overline{K}"']
\end{tikzcd}
\]
In the present situation, this observation yields the following result, which
generalizes the corresponding statement for radical-square-zero algebras;
see \cite[Theorem~6.5]{5}.
\begin{proposition}
Assume that \(\Lambda\) is Iwanaga--Gorenstein. Equip the stable category
\(\Lambda\textup{-}\underline{\mathrm{Gproj}}^{\mathbb Z}\)
with the split triangulated structure whose translation functor is
\(\langle 1\rangle[-1],\)
where \([-1]=\Omega\) is the inverse of the suspension induced by the Frobenius
structure. Then the functor
\[
\overline{K}\colon
\Lambda^{!}\textup{-Fcp}^{\mathbb Z}/\Lambda^{!}\textup{-gmod}
\longrightarrow
\Lambda\textup{-}\underline{\mathrm{Gproj}}^{\mathbb Z}
\]
is a triangulated equivalence.
\end{proposition}
\section{Applications}
In this final subsection, we apply the preceding results to investigate certain
nonstandard \(t\)\nobreakdash-structures on the triangulated categories
\(\mathsf{D}^{b}(\Lambda\textup{-gmod})\) and
\(\mathsf{D}_{\mathrm{sg}}(\Lambda\textup{-gmod})\). We then show that, for
quadratic monomial algebras, the finitistic dimension admits a natural
Koszul-dual estimate governed by finite paths in \(\Lambda^{!}\).

\subsection{Graded \texorpdfstring{$t$}{t}-structures}

We now show that the triangulated categories
\(\mathsf{D}^{b}(\Lambda\textup{-gmod})\) and
\(\mathsf{D}_{\mathrm{sg}}(\Lambda\textup{-gmod})\)
admit nonstandard \(t\)\nobreakdash-structures induced, via graded Koszul duality, from the categories
\(\mathsf{D}^{b}(\Lambda^{!}\textup{-Fcp}^{\mathbb{Z}})\) and
\(\mathsf{D}^{b}(\Lambda^{!}\textup{-Fcp}^{\mathbb{Z}}/\Lambda^{!}\textup{-gmod})\),
respectively.

In the derived setting for quadratic monomial algebras, the \(t\)\nobreakdash-structures constructed here should be viewed as the counterparts of the \(t\)\nobreakdash-structures obtained by transporting the standard \(t\)\nobreakdash-structure along Koszul duality of Beilinson, Ginzburg, and Soergel~\cite[Section~2.13]{3}.

\medskip

For completeness, we briefly recall the notion of the standard \(t\)\nobreakdash-structure on the bounded derived category of an abelian category and its heart.

Let \(\mathcal{A}\) be an abelian category. The bounded derived category \(\mathsf{D}^{b}(\mathcal{A})\) carries a canonical \(t\)\nobreakdash-structure defined by
\[
\mathsf{D}^{\le 0}
:=
\bigl\{
X^{\bullet} \in \mathsf{D}^{b}(\mathcal{A})
\;\big|\;
H^{i}(X^{\bullet})=0 \text{ for all } i>0
\bigr\},
\]
\[
\mathsf{D}^{\ge 0}
:=
\bigl\{
X^{\bullet} \in \mathsf{D}^{b}(\mathcal{A})
\;\big|\;
H^{i}(X^{\bullet})=0 \text{ for all } i<0
\bigr\}.
\]
The pair \((\mathsf{D}^{\le 0}, \mathsf{D}^{\ge 0})\) defines a \(t\)\nobreakdash-structure on \(\mathsf{D}^{b}(\mathcal{A})\), referred to as the \emph{standard \(t\)\nobreakdash-structure}.

The heart of this \(t\)\nobreakdash-structure is the full subcategory
\[
\mathsf{D}^{\heartsuit}
:=
\mathsf{D}^{\le 0} \cap \mathsf{D}^{\ge 0},
\]
which is canonically equivalent to the abelian category \(\mathcal{A}\).

Before proving the next theorem, we shall need the following lemma.

\begin{lemma}
Let \(X^{\bullet} \in \mathsf{D}^{b}\!\bigl(\Lambda^{!}\textup{-Fcp}^{\mathbb{Z}}\bigr)\), and let \(I_{x}^{!}\) be the indecomposable graded injective \(\Lambda^{!}\)-module corresponding to \(x \in Q_{0}\). Then, for all \(m,i \in \mathbb{Z}\), there is a natural isomorphism
\[
\operatorname{Hom}_{\mathsf{D}^{b}(\Lambda^{!}\textup{-Fcp}^{\mathbb{Z}})}
\bigl(X^{\bullet}, I_{x}^{!}\langle m\rangle[-i]\bigr)
\cong
\operatorname{Hom}_{\Lambda^{!}\textup{-Fcp}^{\mathbb{Z}}}
\bigl(H^{i}(X^{\bullet}), I_{x}^{!}\langle m\rangle\bigr).
\]
\end{lemma}
\begin{proof}
Put \(I=I_x^{!}\langle m\rangle\). Since \(I\) is injective, the morphisms into \(I[-i]\) in the bounded
derived category are computed in the homotopy category:
\[
\operatorname{Hom}_{\mathsf D^b}
\bigl(X^\bullet,I[-i]\bigr)
=
\operatorname{Hom}_{\mathsf K^b}
\bigl(X^\bullet,I[-i]\bigr).
\]

A chain map \(f^\bullet:X^\bullet\to I[-i]\) is determined by a morphism
\(f^i:X^i\to I\) satisfying \(f^i d_X^{i-1}=0\). Thus \(f^i\) factors through
\(X^i/\operatorname{Im}d_X^{i-1}\). Two such maps are homotopic if and only if
their difference factors through \(d_X^i:X^i\to X^{i+1}\), equivalently through
\(\operatorname{Im}d_X^i\). Therefore
\[
\operatorname{Hom}_{\mathsf K^b}
\bigl(X^\bullet,I[-i]\bigr)
\cong
\frac{
\operatorname{Hom}\bigl(X^i/\operatorname{Im}d_X^{i-1},I\bigr)
}{
\operatorname{Hom}\bigl(\operatorname{Im}d_X^i,I\bigr)
}.
\]

Now there is a short exact sequence
\[
0\longrightarrow H^i(X^\bullet)
\longrightarrow X^i/\operatorname{Im}d_X^{i-1}
\longrightarrow \operatorname{Im}d_X^i
\longrightarrow 0.
\]
Since \(I\) is injective, applying \(\operatorname{Hom}(-,I)\) gives an exact
sequence
\[
0\longrightarrow
\operatorname{Hom}(\operatorname{Im}d_X^i,I)
\longrightarrow
\operatorname{Hom}\bigl(X^i/\operatorname{Im}d_X^{i-1},I\bigr)
\longrightarrow
\operatorname{Hom}\bigl(H^i(X^\bullet),I\bigr)
\longrightarrow 0.
\]
Hence
\[
\operatorname{Hom}_{\mathsf D^b}
\bigl(X^\bullet,I[-i]\bigr)
\cong
\operatorname{Hom}\bigl(H^i(X^\bullet),I\bigr).
\]
\end{proof}
We are now ready to show how Koszul duality transports the standard \(t\)-structure to nonstandard \(t\)-structures on \(\mathsf{D}^{b}(\Lambda\textup{-gmod})\) and on  \(\mathsf{D}_{\mathrm{sg}}(\Lambda\textup{-gmod})\).
\begin{theorem}
Let
\[
\mathfrak{F}\colon
\mathsf{D}^{b}\bigl(\Lambda^{!}\textup{-Fcp}^{\mathbb{Z}}\bigr)
\xrightarrow{\ \sim\ }
\mathsf{D}^{b}(\Lambda\textup{-gmod})
\]
be the graded derived Koszul duality functor. Define full subcategories of
\(\mathsf{D}^{b}(\Lambda\textup{-gmod})\) by
\[
\begin{aligned}
\mathcal{D}^{\le 0}_{\mathfrak{F}}
&:=
\mathfrak{F}\Bigl(
\mathsf{D}^{\le 0}\bigl(\Lambda^{!}\textup{-Fcp}^{\mathbb{Z}}\bigr)
\Bigr),\\
\mathcal{D}^{\ge 0}_{\mathfrak{F}}
&:=
\mathfrak{F}\Bigl(
\mathsf{D}^{\ge 0}\bigl(\Lambda^{!}\textup{-Fcp}^{\mathbb{Z}}\bigr)
\Bigr).
\end{aligned}
\]
Then \((\mathcal{D}^{\le 0}_{\mathfrak{F}},\mathcal{D}^{\ge 0}_{\mathfrak{F}})\) is a \(t\)-structure on
\(\mathsf{D}^{b}(\Lambda\textup{-gmod})\). Its heart is equivalent to the abelian category
\[
\mathcal{H}_{\mathfrak{F}}
\;\simeq\;
\Lambda^{!}\textup{-Fcp}^{\mathbb{Z}},
\]
and identifies with the full subcategory of linear complexes with bounded cohomology.

\medskip

Moreover, this \(t\)-structure admits the following intrinsic description:
\[
\mathcal{D}^{\le 0}_{\mathfrak{F}}
=
\left\{
Y^{\bullet} \in \mathsf{D}^{b}(\Lambda\textup{-gmod})
\;\middle|\;
\begin{array}{l}
\operatorname{Hom}_{\mathsf{D}^{b}(\Lambda\textup{-gmod})}
\bigl(Y^{\bullet}, S_x\langle m\rangle[i-m]\bigr)=0\\
\qquad \text{for all } i>0,\; x\in Q_0,\; m\in\mathbb{Z}
\end{array}
\right\},
\]
\[
\mathcal{D}^{\ge 0}_{\mathfrak{F}}
=
\left\{
Y^{\bullet} \in \mathsf{D}^{b}(\Lambda\textup{-gmod})
\;\middle|\;
\begin{array}{l}
\operatorname{Hom}_{\mathsf{D}^{b}(\Lambda\textup{-gmod})}
\bigl(Y^{\bullet}, S_x\langle m\rangle[i-m]\bigr)=0\\
\qquad \text{for all } i<0,\; x\in Q_0,\; m\in\mathbb{Z}
\end{array}
\right\}.
\]

\medskip

Furthermore, the same functor \(\mathfrak{F}\) induces a triangulated equivalence
\[
\mathfrak{F}\colon
\mathsf{D}^{b}\bigl(
\Lambda^{!}\textup{-Fcp}^{\mathbb{Z}} \big/ \Lambda^{!}\textup{-gmod}
\bigr)
\xrightarrow{\ \sim\ }
\mathsf{D}_{\mathrm{sg}}(\Lambda\textup{-gmod}),
\]
and the full subcategories
\[
\begin{aligned}
\mathcal{D}^{\le 0}_{\mathfrak{F}}
&:=
\left\{
Y^{\bullet}\in
\mathsf{D}_{\mathrm{sg}}(\Lambda\textup{-gmod})
\;\middle|\;
Y^{\bullet}\simeq
\bigoplus_{i\le 0}
\mathfrak{F}(M_i)[-i],
\quad
M_i\in
\Lambda^{!}\textup{-Fcp}^{\mathbb Z}/\Lambda^{!}\textup{-gmod}
\right\},\\
\mathcal{D}^{\ge 0}_{\mathfrak{F}}
&:=
\left\{
Y^{\bullet}\in
\mathsf{D}_{\mathrm{sg}}(\Lambda\textup{-gmod})
\;\middle|\;
Y^{\bullet}\simeq
\bigoplus_{i\ge 0}
\mathfrak{F}(M_i)[-i],
\quad
M_i\in
\Lambda^{!}\textup{-Fcp}^{\mathbb Z}/\Lambda^{!}\textup{-gmod}
\right\}
\end{aligned}
\]
define a \(t\)-structure on
\(\mathsf{D}_{\mathrm{sg}}(\Lambda\textup{-gmod})\) such that
 each object \(M_i\) is represented by a colinear graded
\(\Lambda^{!}\)-module, and
\(\mathfrak{F}(M_i)\) is a linear projective resolution of a linear \(\Lambda\)-modules.

Its heart is equivalent to
\[
\mathcal{H}_{\mathfrak{F}}
\;\simeq\;
\Lambda^{!}\textup{-Fcp}^{\mathbb{Z}}
\big/
\Lambda^{!}\textup{-gmod},
\]
which is a hereditary abelian category. Moreover, this heart identifies concretely with the
full subcategory consisting of stalk complexes of linear \(\Lambda\)-modules.
\end{theorem}

\begin{proof}
Since \(\mathfrak{F}\) is a triangulated equivalence, it transports the standard \(t\)-structure on
\(\mathsf{D}^{b}\bigl(\Lambda^{!}\textup{-Fcp}^{\mathbb{Z}}\bigr)\) to a \(t\)-structure on
\(\mathsf{D}^{b}(\Lambda\textup{-gmod})\). By construction, the aisle and coaisle are given by
\[
\mathcal{D}^{\le 0}_{\mathfrak{F}}
=
\mathfrak{F}\bigl(\mathsf{D}^{\le 0}(\Lambda^{!}\textup{-Fcp}^{\mathbb{Z}})\bigr),
\qquad
\mathcal{D}^{\ge 0}_{\mathfrak{F}}
=
\mathfrak{F}\bigl(\mathsf{D}^{\ge 0}(\Lambda^{!}\textup{-Fcp}^{\mathbb{Z}})\bigr),
\]
and hence define a \(t\)-structure.

On the other hand, by lemma 4.1, for any \(X^{\bullet}\in \mathsf{D}^{b}(\Lambda\textup{-gmod})\), one has
\(H^{i}(X^{\bullet})=0 \text{ for all } i>0 \ (\text{resp. } i<0)\)
if and only if
\[
\operatorname{Hom}_{\mathsf{D}^{b}(\Lambda^{!}\textup{-Fcp}^{\mathbb{Z}})}
\bigl(X^{\bullet}, I_x^{!}\langle m\rangle[i]\bigr)=0
\quad \text{for all } i>0 \ (\text{resp. } i<0),\ x\in Q_0,\ m\in\mathbb{Z}.
\]
Transporting this characterization via the equivalence \(\mathfrak{F}\) yields the explicit description of the subcategories \(\mathcal{D}^{\le 0}_{\mathfrak{F}}\) and \(\mathcal{D}^{\ge 0}_{\mathfrak{F}}\).

The heart is the image of the heart of the standard \(t\)-structure, namely
\[
\mathfrak{F}\bigl(\Lambda^{!}\textup{-Fcp}^{\mathbb{Z}}\bigr),
\]
which, by graded derived Koszul duality, identifies with the category of linear complexes with bounded cohomology.

The statement for the singularity category follows from Proposition~3.1 and Corollary 3.14 together with graded singular Koszul duality.
\end{proof}
\begin{remark}
The dual statement of Proposition~3.17 shows that
\(\mathsf{D}^{b}\bigl(\Lambda^{!}\textup{-Fcp}^{\mathbb{Z}}\bigr)\)
also admits a nonstandard \(t\)-structure induced from the standard
\(t\)-structure on
\(\mathsf{D}^{b}(\Lambda\textup{-gmod})\)
via the quasi-inverse functor \(\mathfrak{G}\). Its heart identifies with the
category of bounded colinear complexes.
\end{remark}

\subsection{The finitistic dimension}

In this final subsection, we study the finitistic dimension from the viewpoint
of Koszul duality. Recall that, for a finite-dimensional algebra \(\Lambda\),
the finitistic dimension is defined by
\[
\operatorname{fin.dim}\Lambda
=
\sup\{
\operatorname{pd}_{\Lambda} M
\mid
M\in \Lambda\textup{-mod},\
\operatorname{pd}_{\Lambda}M<\infty
\}.
\]
Thus \(\operatorname{fin.dim}\Lambda\) measures the largest projective dimension
occurring among finitely generated \(\Lambda\)-modules of finite projective
dimension. The Finitistic Dimension Conjecture asserts that this invariant is
finite for every finite-dimensional algebra. This conjecture, publicized by
Bass, remains one of the central open problems in the homological theory of
finite-dimensional algebras.

The conjecture is known to hold for monomial algebras. More precisely,
if \(\Lambda=kQ/I\), where \(I\) is generated by paths, then
\(\operatorname{fin.dim}\Lambda<\infty\). This was first proved by Green,
Kirkman, and Kuzmanovich~\cite{13}; subsequent refinements gave simplified
proofs and explicit bounds~\cite{6,18,32}. Nevertheless, the conjecture remains
open in substantially broader classes of finite-dimensional algebras, including
the class of Koszul algebras.

We focus here on finite-dimensional quadratic monomial algebras. Our aim is to
show that, in this setting, the finitistic dimension can be bounded in terms of
a simple invariant of the quadratic dual algebra \(\Lambda^{!}\). The key point
is that, for monomial relation algebras, second syzygies admit an explicit
description in terms of principal left ideals generated by arrows; see
Lemma~3.2 and~\cite{32}.

Koszul duality provides the bridge between these two descriptions. Indeed, the
Koszul functor sends graded \(\Lambda^{!}\)-modules to complexes of finitely
generated projective \(\Lambda\)-modules, and the graded support of suitable
\(\Lambda^{!}\)-modules controls the length of the corresponding projective
complexes over \(\Lambda\). Thus, instead of estimating projective dimensions
over \(\Lambda\) directly, one may bound them by studying the finite lengths of
nonzero paths in the Koszul dual algebra \(\Lambda^{!}\). For a finite path
\(p^{\operatorname{op}}\) in \(\Lambda^{!}\), we denote by
\(\ell(p^{\operatorname{op}})\) its length. This yields the following estimate
for the finitistic dimension, which can be computed from the lengths of finite
paths in \(\Lambda^{!}\).

\begin{theorem}
Let \(\Lambda=kQ/I\) be a finite-dimensional quadratic monomial algebra. Then
\[
\operatorname{fin.dim}\Lambda
\leq
2+
\sup\{
\ell(p^{\operatorname{op}})
\mid
p^{\operatorname{op}} \textup{ is a nonzero path of finite length in }
\Lambda^{!}
\}.
\]
In particular, if \(\operatorname{gl.dim}\Lambda<\infty\), then
\[
\operatorname{gl.dim}\Lambda
=
\max\{\, n\in \mathbb{N}\mid \Lambda^{!}_{n}\neq 0\,\}.
\]
\end{theorem}

\begin{proof} The first assertion follows immediately from the fact that the second syzygy can be written as a finite direct sum of principal ideals generated by arrows, together with Lemma~3.2 and Proposition~3.1. For the second assertion, the Koszul functor sends each indecomposable injective \(\Lambda^{!}\)-module \(I_x^{!}\) to the minimal graded projective resolution of the simple \(\Lambda\)-module \(S_x\). Therefore, the global dimension of \(\Lambda\) is determined by the largest nonzero graded component of \(\Lambda^{!}\). \end{proof}
\begin{remark}
This theorem illustrates how Koszul duality translates finiteness properties of
the Koszul dual  \(\Lambda^{!}\) into homological bounds for
\(\Lambda\). It suggests that analogous methods may be useful in the study of
graded finitistic dimensions for more general finite-dimensional Koszul
algebras.
\end{remark}

\noindent
\textit{E-mail address:} \texttt{alesm.bouhada@gmail.com}
\end{document}